\documentclass{amsart}
\usepackage{graphicx}
\usepackage{amscd}
\usepackage{amsmath}
\usepackage{amsfonts}
\usepackage{amssymb}
\newtheorem{theorem}{Theorem}
\theoremstyle{plain}

\newtheorem{corollary}{Corollary}

\newtheorem{definition}{Definition}
\newtheorem{example}{Example}

\newtheorem{proposition}{Proposition}
\newtheorem{remark}{Remark}

\numberwithin{equation}{section}

\begin{document}
\title[dimensionless Sobolev inequalities]{Integral isoperimetric transference and dimensionless Sobolev inequalities}
\author{Joaquim Mart\'{i}n$^{\ast}$}
\address{Department of Mathematics\\
Universitat Aut\`{o}noma de Barcelona}
\email{jmartin@mat.uab.cat}
\author{Mario Milman**}
\email{extrapol@bellsouth.net}
\urladdr{https://sites.google.com/site/mariomilman}
\thanks{$^{\ast}$Partially supported in part by Grants MTM2010-14946, MTM-2010-16232.}
\thanks{**This work was partially supported by a grant from the Simons Foundation
(\#207929 to Mario Milman).}
\thanks{This paper is in final form and no version of it will be submitted for
publication elsewhere.}

\begin{abstract}
We introduce the concept of Gaussian integral isoperimetric transfererence and show
how it can be applied to obtain a new class of sharp Sobolev-Poincar\'{e}
inequalities with constants independent of the dimension. In the special case
of $L^{q}$ spaces on the unit $n-$dimensional cube our results extend the
recent inequalities that were obtained in \cite{FKS} using extrapolation.
\end{abstract}\maketitle

\section{Introduction}

Let $Q_{n}$ be the open unit cube in $\mathbb{R}^{n},$ let $1\leq q<n$ be
fixed$,$ and let $\frac{1}{p}=\frac{1}{q}-\frac{1}{n};$ a special case of the
classical (homogeneous) Sobolev inequality states that for all $f\in
C_{0}^{\infty}(Q_{n}),$%
\[
\left\|  f\right\|  _{L^{p}(Q_{n})}\leq c(n,q)\left\|  \nabla f\right\|
_{L^{q}(Q_{n})}.
\]
It is well known (cf. \cite{alv}, \cite{ta}) that $c(n,q)\asymp c(q)n^{-\frac
{1}{2}},$ with $c(q)$ independent of $n.$ Triebel \cite{tri}, \cite{tri1},
recently suggested the problem of finding dimension free\footnote{Dealing with
a large number of variables (high dimensionality) gives rise to a set of
specific problems and issues in analysis, e.g. in approximation theory,
numerical analysis, optimization.. (cf. \cite{griebel}).} Sobolev
inequalities, at least in what concerns the constants involved, by means of
replacing the power gain of integrability on the left hand side by a (smaller)
logarithmic gain. This, indeed, can be achieved by different methods (e.g. by
transference from Gaussian inequalities via isoperimetry and symmetrization
(cf. \cite{mamiadv}), by direct transference from Log Sobolev inequalities
\cite{krb1}, by extrapolation, using weighted norms inequalities (cf.
\cite{krb1}, \cite{krb2}), etc.). The sharper known results in this direction
give inequalities of the form%
\begin{equation}
\left\|  f\right\|  _{L^{q}(LogL)^{q/2}(Q_{n})}\leq C(q)\left\|  \nabla
f\right\|  _{L^{q}(Q_{n})},\text{ }f\in C_{0}^{\infty}(Q_{n}),
\label{launoprima}%
\end{equation}
with $C(q)$ independent of the dimension, and a corresponding non-homogeneous
version: for $f\in W^{1,1}(Q_{n}),$ we have%
\begin{equation}
\left\|  f\right\|  _{L^{q}(LogL)^{q/2}(Q_{n})}\leq C(q)\left(  \left\|
f\right\|  _{L^{q}(Q_{n})}+\left\|  \nabla f\right\|  _{L^{q}(Q_{n})}\right)
. \label{launo}%
\end{equation}

The proof via Gaussian isoperimetric transference obtained in \cite{mamiadv}
explains the presence of the factor $\frac{1}{2}$ in the logarithmic exponent.
To summarize, sacrificing the power gain of integrability of the classical
Sobolev inequality for the log Sobolev gain of integrability associated with
Gaussian measure, we obtain Sobolev inequalities on the unit cube $Q_{n}$ with
constants independent of the dimension.

It is not hard to see that the logarithmic inequalities (\ref{launo}),
(\ref{launoprima}), give the optimal results within the class of
$L^{q}(LogL)^{r}$ spaces. However, Fiorenza-Krbec-Schmeisser \cite{FKS} have
recently shown that, on the larger class of rearrangement invariant spaces,
the optimal inequality with dimensionless constants corresponding to
(\ref{launoprima}) is given by%
\begin{equation}
\left\|  f\right\|  _{L_{(q,\frac{q^{\prime}}{2}}(Q_{n})}\leq c(q)\left\|
\nabla f\right\|  _{L^{q}(Q_{n})},\text{ }f\in C_{0}^{\infty}(Q_{n}),
\label{lafks}%
\end{equation}
where the space $L_{(q,\frac{q^{\prime}}{2}}(Q_{n})$ is the so called `small'
Lebesgue space introduced\footnote{As the dual of the `grand' Lebesgue spaces,
introduced by Iwaniec and Sbordone \cite{iw}.} by Fiorenza \cite{fio}. The
space $L_{(q,\frac{q^{\prime}}{2}}(Q_{n})$ is defined by means of the
following norm%
\begin{equation}
\left\|  f\right\|  _{L_{(q,q^{\prime}}(Q_{n})}=\inf_{f=\sum f_{j}}\left(
\sum_{j}\inf_{0<\varepsilon<q^{\prime}-1}\varepsilon^{-\frac{q^{\prime}%
/2}{q^{\prime}-\varepsilon}}\left\|  f_{j}\right\|  _{L^{(q^{\prime
}-\varepsilon)^{\prime}}(Q_{n})}\right)  , \label{lasmall1}%
\end{equation}
where as usual $q^{\prime}=q/(q-1).$ Moreover, $L_{(q,\frac{q^{\prime}}{2}%
}(Q_{n})$ can be characterized as an extrapolation space in the sense of
Karadzhov-Milman \cite{kami}; therefore, its norm can be computed explicitly
(cf. \cite{fioka}),
\begin{equation}
\left\|  f\right\|  _{L_{(q,q^{\prime}}(Q_{n})}\approx\int_{0}^{1}\left(
\int_{0}^{t}f^{\ast}(s)^{q}ds\right)  ^{1/q}\frac{dt}{t(\log\frac{1}%
{t})^{\frac{1}{2}}}. \label{lasmall2}%
\end{equation}
As usual, the symbol $f\approx g$ will indicate the existence of a universal
constant $C>0$ (independent of all parameters involved) so that $(1/C)f\leq
g\leq C\,f$, while the symbol $f\preceq g$ means that for a suitable constant
$C,$ $f\leq C\,g,$ and likewise $f\succeq g$ means that $f\geq Cg.$

It is not difficult to see that (cf. \cite{FKS}, see also (\ref{laursa})
below)
\[
L_{(q,q^{\prime}}(Q_{n})\varsubsetneq L^{q}(LogL)^{q/2}(Q_{n}).
\]

The proof of the inequality (\ref{lafks}) given in \cite{FKS} depends on
extrapolation and is accomplished using (\ref{lasmall1}) or (\ref{lasmall2}),
and the explicit form of the Sobolev embedding constant$.$

In this paper we investigate the connection of inequalities of the form
(\ref{lafks}) to the isoperimetric inequality and show a new associated
transference principle. We work on metric probability spaces and study a class
of Sobolev inequalities which include (\ref{lafks}), which are valid if the
isoperimetric profile satisfies a suitable integrability condition. In
particular, inequalities of the form (\ref{lafks}) are connected with what
could be termed a Gaussian transfer condition, as we now explain.

Underlying our method are certain pointwise rearrangement inequalities for
Sobolev functions which are associated with the isoperimetric profile of a
given geometry (cf. \cite{mamiadv}, \cite{mamiast}). Using these pointwise
inequalities we will obtain, by integral transference, inequalities that are
stronger than the usual transferred (log) Sobolev inequalities, while still
preserving the dimensionless of the constants involved. More generally, we are
able to give a unified approach to a class of dimensionless inequalities for
different geometries, that hold within the class of general rearrangement
invariant spaces, and as we shall show elsewhere, encompass
fractional\footnote{The basic fractional inequalities that underlie our
analysis were obtained in \cite{mamiast}.} inequalities as well.

To discuss the results of the paper in more detail it will be useful to recall
first a version of the transference principle developed in \cite{mamiadv}. It
was shown in \cite{mamiadv} that, for a large class of connected metric
probability spaces $(\Omega,d,\mu),$ with associated isoperimetric
profile\footnote{See Section \ref{back} below; in particular we assume that
\ $I(t)$ is concave and symmetric about $1/2.$} $I=I_{(\Omega,\mu)}$ and for
rearrangement invariant spaces $\bar{X}=$ $\bar{X}(0,1)$ on $(0,1),$ that in a
suitable technical sense are `away' from $L^{1},$ we have the following
Sobolev-Poincar\'{e} inequality: For all Lipchitz functions on $\Omega,$
\begin{equation}
\left\|  \left(  f^{\ast\ast}(t)-f^{\ast}(t)\right)  \frac{I(t)}{t}\right\|
_{\bar{X}}\leq c\left\|  \left|  \nabla f\right|  ^{\ast}\right\|  _{\bar{X}},
\label{lasobo}%
\end{equation}
where $f^{\ast}$ is the non-increasing rearrangement\footnote{Precise
definitions and properties of rearrangements and related topics coming into
play in this section are contained in Section \ref{back}} of $f,$ $f^{\ast
\ast}(t)=\frac{1}{t}\int_{0}^{t}f^{\ast}(s)ds,$ and $c$ is a universal
constant that depends only on $\bar{X}.$

The inequality (\ref{lasobo}) is best possible\footnote{as far as the
condition on the left hand side (or target space).}, but the isoperimetric
profile involved may depend on the dimension, and thus there could be a
dimensional dependency in the constants. A natural method to obtain
dimensionless inequalities is to weaken the norm inequality (\ref{lasobo}) via
transference. For example, if $(\Omega,d,\mu)$ is of ``Gaussian isoperimetric
type'', i.e. if for some universal constant independent of the dimension, it
holds
\begin{equation}
I_{(\Omega,\mu)}(t)\succeq t\left(  \log\frac{1}{t}\right)  ^{\frac{1}{2}%
},\text{ on }\left(  0,\frac{1}{2}\right)  ; \label{laiso}%
\end{equation}
then the Gaussian log Sobolev inequalities can be transferred to
$(\Omega,d,\mu)$ with constants independent of the dimension. The method of
\cite{mamiadv} simply expresses the fact that from (\ref{lasobo}) and
(\ref{laiso}) we can see that for all r.i. spaces away from $L^{1},$ we have%
\begin{equation}
\left\|  \left(  f^{\ast\ast}(t)-f^{\ast}(t)\right)  \left(  \log\frac{1}%
{t}\right)  ^{\frac{1}{2}}\right\|  _{\bar{X}}\leq c\left\|  \left|  \nabla
f\right|  ^{\ast}\right\|  _{\bar{X}}. \label{lavieja}%
\end{equation}
In particular, the Euclidean unit $n-$dimensional cube $Q_{n}$ is of Gaussian
type, with constant $1$ (cf. \cite{Ros}), therefore (\ref{lavieja}) for
$\bar{X}=L^{q}$ gives (\ref{launo}). More generally, these inequalities can be
reformulated as
\[
\left\|  \left(  f^{\ast\ast}(t)-f^{\ast}(t)\right)  G_{\infty}(t)\right\|
_{\bar{X}}\leq c\left\|  G_{\infty}\frac{t}{I(t)}\right\|  _{L^{\infty
}(0,\frac{1}{2})}\left\|  \left|  \nabla f\right|  ^{\ast}\right\|  _{\bar{X}%
},
\]
where (\ref{lavieja}) corresponds to the choice $G_{\infty}(t)=(\log\frac
{1}{t})^{\frac{1}{2}}.$

The inequality (\ref{lasobo}) was proved under the assumption that
$\frac{I(t)}{t}$ decreases$,$ an assumption we shall keep in this paper. To
proceed further we note that the left hand side of (\ref{lasobo}) can be
minorized as follows,%
\begin{align*}
\left\|  \left(  f^{\ast\ast}(\cdot)-f^{\ast}(\cdot)\right)  \frac{I(\cdot
)}{(\cdot)}\right\|  _{\bar{X}}  &  \geq\left\|  \left(  f^{\ast\ast}%
(\cdot)-f^{\ast}(\cdot)\right)  \chi_{(0,t)}(\cdot)\frac{I(\cdot)}{(\cdot
)}\right\|  _{\bar{X}}\\
&  \geq\left\|  \left(  f^{\ast\ast}(\cdot)-f^{\ast}(\cdot)\right)
\chi_{(0,t)}(\cdot)\right\|  _{\bar{X}}\frac{I(t)}{t}.
\end{align*}
Therefore, we have%
\begin{equation}
\left\|  \left(  f^{\ast\ast}(\cdot)-f^{\ast}(\cdot)\right)  \chi
_{(0,t)}(\cdot)\right\|  _{\bar{X}}\leq c\frac{t}{I(t)}\left\|  \left|  \nabla
f\right|  ^{\ast}\right\|  _{\bar{X}}. \label{larusa}%
\end{equation}
Now, if $G$ is such that $\left(  \int_{0}^{1}G(t)\frac{t}{I(t)}dt\right)
<\infty,$ it follows immediately from (\ref{larusa}) that%
\[
\int_{0}^{1}\left\|  \left(  f^{\ast\ast}(\cdot)-f^{\ast}(\cdot)\right)
\chi_{(0,t)}(\cdot)\right\|  _{\bar{X}}G(t)dt\leq C\left(  \int_{0}%
^{1}G(t)\frac{t}{I(t)}dt\right)  \left\|  \left|  \nabla f\right|  ^{\ast
}\right\|  _{\bar{X}}.
\]
For example, for $G(t)=\frac{1}{t\left(  \log\frac{1}{t}\right)  ^{\frac{1}%
{2}}},$ we are able to control the functional
\[
\int_{0}^{1}\left\|  \left(  f^{\ast\ast}(s)-f^{\ast}(s)\right)  \chi
_{(0,t)}(s)\right\|  _{\bar{X}}\frac{dt}{t\left(  \log\frac{1}{t}\right)
^{\frac{1}{2}}},
\]
as long as the following (stronger) Gaussian isoperimetric transference
condition is satisfied,
\begin{equation}
\int_{0}^{1}\frac{dt}{I(t)(\log\frac{1}{t})^{\frac{1}{2}}}<\infty.
\label{laconversa}%
\end{equation}
To compare the different results let us note that by the triangle inequality
we have%
\begin{align*}
I  &  =\int_{0}^{1}\left\|  (f^{\ast\ast}(\cdot)-f^{\ast}(\cdot))\chi
_{(0,t)}(\cdot)\right\|  _{\bar{X}}\frac{dt}{t\left(  \log\frac{1}{t}\right)
^{1/2}}\\
&  =\int_{0}^{1}\left\|  (f^{\ast\ast}(\cdot)-f^{\ast}(\cdot))\chi
_{(0,t)}(\cdot)\frac{1}{t\left(  \log\frac{1}{t}\right)  ^{1/2}}\right\|
_{\bar{X}}dt\\
&  =\int_{0}^{1}\left\|  h(\cdot,t)\right\|  _{\bar{X}}dt,\text{ where
}h(s,t)=(f^{\ast\ast}(s)-f^{\ast}(s))\chi_{(0,t)}(s)\frac{1}{t\left(
\log\frac{1}{t}\right)  }\\
&  \geq\left\|  \int_{0}^{1}h(s,t)dt\right\|  _{\bar{X}}\\
&  =\left\|  (f^{\ast\ast}(\cdot)-f^{\ast}(\cdot))\int_{s}^{1}\frac
{1}{t\left(  \log\frac{1}{t}\right)  ^{1/2}}dt\right\|  _{\bar{X}}\\
&  \geq2\left\|  (f^{\ast\ast}(\cdot)-f^{\ast}(\cdot))\left(  \log(\frac
{1}{\cdot})\right)  ^{1/2}\right\|  _{\bar{X}}%
\end{align*}
Thus, if (\ref{laconversa}) holds, we have\footnote{which of course is still
weaker than the optimal inequality (\ref{lasobo}).}
\begin{align}
\left\|  (f^{\ast\ast}(\cdot)-f^{\ast}(\cdot))\left(  \log(\frac{1}{\cdot
})\right)  ^{1/2}\right\|  _{\bar{X}}  &  \leq c\int_{0}^{1}\left\|  \left(
f^{\ast\ast}(s)-f^{\ast}(s)\right)  \chi_{(0,t)}(s)\right\|  _{\bar{X}}%
\frac{dt}{t\left(  \log\frac{1}{t}\right)  ^{\frac{1}{2}}}\nonumber\\
&  \leq c\left(  \int_{0}^{1}\frac{dt}{I(t)(\log\frac{1}{t})^{\frac{1}{2}}%
}\right)  \left\|  \left|  \nabla f\right|  ^{\ast}\right\|  _{\bar{X}}.
\label{laursa}%
\end{align}

At this point we can explain what could be gained from our efforts. For
suitable domains, e.g. for the Euclidean unit cubes $Q_{n},$ or, more
generally, for some other classes of $n-$Euclidean domains, the corresponding
integral conditions can be estimated by a constant independent of the
dimension as follows
\begin{equation}
\sup_{n}\int_{0}^{1}\frac{dt}{I_{n}(t)(\log\frac{1}{t})^{\frac{1}{2}}}<\infty.
\label{laconstante}%
\end{equation}
When (\ref{laconstante}) holds the resulting inequalities we have thus
obtained are both dimensionless and stronger than the ones that could be
derived via the `pointwise' Gaussian transference condition.

The restriction that the space $\bar{X}$ must be `away from $L^{1}$' can be
removed using the generalized P\'{o}lya-Szeg\"{o} principle of \cite{mamiadv}:
Under the assumption that $\bar{X}$ is `away from $L^{\infty}$', we can
replace (\ref{larusa}) by (cf. Theorem \ref{main} below),%
\[
\left\|  \left(  f^{\ast}(\cdot)-f^{\ast}(t)\right)  \chi_{\lbrack0,t)}%
(\cdot)\right\|  _{\bar{X}}\leq c\frac{t}{I(t)}\left\|  Q\right\|  _{\bar
{X}\rightarrow\bar{X}}\left\|  \left|  \nabla f\right|  ^{\ast}\right\|
_{\bar{X}},
\]
where $Qg(t)=\int_{t}^{1}g(s)\frac{ds}{s},$ and from this point the analysis
can proceed along the lines outlined above.

We will also show a partial converse of this result, for example, the
inequality
\begin{equation}
\int_{0}^{1}\left\|  \left(  f^{\ast}(\cdot)-f^{\ast}(t)\right)  \chi
_{\lbrack0,t)}(\cdot)\right\|  _{L^{1}}\frac{dt}{t(\log\frac{1}{t})^{1/2}}\leq
C\left\|  \left|  \nabla f\right|  ^{\ast}\right\|  _{L^{1}}, \label{lapro}%
\end{equation}
implies that the profile of $(\Omega,d,\mu)$ must satisfy a Gaussian type
condition (cf. Corollary \ref{ma1} below)
\[
\frac{t}{I(t)}\int_{t}^{1}\frac{dt}{t(\log\frac{1}{t})^{\frac{1}{2}}}\leq C.
\]

The paper is organized as follows. In Section \ref{back} we collect the
necessary information concerning symmetrizations, isoperimetric profiles and
function spaces considered in this paper; the main inequalities of this paper
are proved in Section \ref{thema}, while Section \ref{appl} contains
applications to different geometries\footnote{For further possible metric
measure spaces were one could consider applications of our method we refer to
\cite{emi} and the very recent \cite{emiro}.}; in particular, in Subsection
\ref{markao} we show in detail how our approach, in the special case of the
unit cubes $Q_{n}$ and $L^{q}$ spaces, yields (\ref{lafks}).

\section{Background\label{back}}

\subsection{Rearrangements}

Let $\left(  \Omega,d,\mu\right)  $ be a Borel probability metric space. For
measurable functions $u:\Omega\rightarrow\mathbb{R},$ the distribution
function of $u$ is given by
\[
\mu_{u}(t)=\mu\{x\in{\Omega}:u(x)>t\}\text{ \ \ \ \ }(t\in\mathbb{R}).
\]
The \textbf{decreasing rearrangement}\footnote{Note that this notation is
somewhat unconventional. In the literature it is common to denote the
decreasing rearrangement of $\left|  u\right|  $ by $u_{\mu}^{\ast},$ while
here it is denoted by $\left|  u_{\mu}\right|  ^{\ast}$ since we need to
distinguish between the rearrangements of $u$ and $\left|  u\right|  .$ In
particular, the rearrangement of $u$ can be negative. We refer the reader to
\cite{rako} and the references quoted therein for a complete treatment.} of a
function $u$ is the right-continuous non-increasing function from $[0,1)$ into
$\mathbb{R}$ which is equimeasurable with $u.$ It can be defined by the
formula%
\[
u_{\mu}^{\ast}(s)=\inf\{t\geq0:\mu_{u}(t)\leq s\},\text{ \ }s\in\lbrack0,1),
\]
and satisfies
\[
\mu_{u}(t)=\mu\{x\in{\Omega}:u(x)>t\}=m\left\{  s\in\lbrack0,1):u_{\mu}^{\ast
}(s)>t\right\}  \text{\ ,\ }t\in\mathbb{R}\text{,}%
\]
where $m$ denotes the Lebesgue measure on $[0,1).$

The maximal average $u_{\mu}^{\ast\ast}(t)$ is defined by
\[
u_{\mu}^{\ast\ast}(t)=\frac{1}{t}\int_{0}^{t}u_{\mu}^{\ast}(s)ds=\frac{1}%
{t}\sup\left\{  \int_{E}u(s)d\mu:\mu(E)=t\right\}  ,\text{ }t>0.
\]
It follows directly from the definition that $(u+v)_{\mu}^{\ast\ast}(t)\leq
u_{\mu}^{\ast\ast}(t)+v_{\mu}^{\ast\ast}(t),$ moreover, since $u_{\mu}^{\ast}$
is decreasing, it follows that $u_{\mu}^{\ast\ast}$ is also decreasing, and
$u_{\mu}^{\ast}\leq u_{\mu}^{\ast\ast}$.

When the probability we are working with is clear from the context, or when we
are dealing with Lebesgue measure, we may simply write $u^{\ast}$ and
$u^{\ast\ast}$, etc.

\subsection{Isoperimetry}

In what follows we always assume that we work with connected Borel probability
metric spaces $\left(  \Omega,d,\mu\right)  $, which we shall simply refer to
as ``measure probability metric spaces''.

Recall that for a Borel set $A\subset\Omega,$ the \textbf{perimeter} or
\textbf{Minkowski content} of $A$ is defined by
\[
P(A;\Omega)=\lim\inf_{h\rightarrow0}\frac{\mu\left(  A_{h}\right)  -\mu\left(
A\right)  }{h},
\]
where $A_{h}=\left\{  x\in\Omega:d(x,A)<h\right\}  .$

The \textbf{isoperimetric profile} is defined by
\[
I_{\Omega}(s)=I_{(\Omega,d,\mu)}(s)=\inf\left\{  P(A;\Omega):\text{ }%
\mu(A)=s\right\}  ,
\]
i.e. $I_{(\Omega,d,\mu)}:[0,1]\rightarrow\left[  0,\infty\right)  $ is the
pointwise maximal function such that%
\[
P(A;\Omega)\geq I_{\Omega}(\mu(A)),
\]
holds for all Borel sets $A$.\textbf{ }Again, when no confusion arises, we
shall drop the subindex $\Omega$ and simply write $I.$

We will always assume that, for the probability metric spaces $(\Omega,d,\mu)$
under consideration, the associated isoperimetric profile $I_{\Omega}$
satisfies that, $I(0)=0,$ $I$ is continuous, concave and symmetric about
$\frac{1}{2}.$ In many cases it is enough to control an `isoperimetric
estimator', i.e. a function $J:[0,\frac{1}{2}]\rightarrow\left[
0,\infty\right)  $ with the same properties as $I$ and such that
\[
I_{\Omega}(t)\geq J(t),\text{ \ \ }t\in(0,1/2).
\]

For a Lipschitz function $f$ on $\Omega$ (briefly $f\in Lip(\Omega))$ we
define the \textbf{modulus of the gradient} by\footnote{In fact one can define
$\left|  \nabla f\right|  $ for functions $f$ that are Lipschitz on every ball
in $(\Omega,d)$ (cf. \cite[pp. 2, 3]{bobk} for more details).}
\[
|\nabla f(x)|=\limsup_{d(x,y)\rightarrow0}\frac{|f(x)-f(y)|}{d(x,y)}.
\]

We shall now summarize some useful inequalities that relate the isoperimetry
with the rearrangements of Lipschitz functions; for more details we refer to
\cite{mamiadv} and \cite{mamiproc}.

\begin{theorem}
\label{Iso}The following statements are equivalent

\begin{enumerate}
\item  Isoperimetric inequality: $\forall A\subset\Omega,$ Borel set with
\[
P(A;\Omega)\geq I(\mu(A)),
\]

\item  Oscillation inequality: $\forall f\in Lip(\Omega),$%
\begin{equation}
(f_{\mu}^{\ast\ast}(t)-f_{\mu}^{\ast}(t))\frac{I(t)}{t}\leq\frac{1}{t}\int
_{0}^{t}\left|  \nabla f\right|  _{\mu}^{\ast}(s)ds,\text{ \ \ }0<t<1.
\label{reod00}%
\end{equation}

\item  P\'{o}lya-Szeg\"{o} inequality: $\forall f\in Lip(\Omega),$ $f_{\mu
}^{\ast}$ is locally absolutely continuous and satisfies
\begin{equation}
\int_{0}^{t}\left(  \left(  -f_{\mu}^{\ast}\right)  ^{\prime}(\cdot
)I(\cdot)\right)  ^{\ast}(s)ds\leq\int_{0}^{t}\left|  \nabla f\right|  _{\mu
}^{\ast}(s)ds\text{, \ \ \ \ \ \ }0<t<1. \label{aa}%
\end{equation}
(The second rearrangement on the left hand side is with respect to the
Lebesgue measure on $\left[  0,1\right)  $).
\end{enumerate}
\end{theorem}

\begin{remark}
Note that $f\in Lip(\Omega)$ implies that $\left|  f\right|  \in Lip(\Omega)$
and $\left|  \nabla\left|  f\right|  \right|  \leq\left|  \nabla f\right|  ,$
consequently, (\ref{reod00}) and (\ref{aa}) hold for $\left|  f\right|  .$
\end{remark}

\subsection{Rearrangement invariant spaces\label{secc:ri}}

We recall briefly the basic definitions and conventions we use from the theory
of rearrangement-invariant (r.i.) spaces, and refer the reader to \cite{BS}
for a complete treatment.

Let $({\Omega},\mu)$ be a probability measure space. Let $X=X({\Omega})$ be a
Banach function space on $({\Omega},\mu),$ with the Fatou
property\footnote{This means that if $f_{n}\geq0,$ and $f_{n}\uparrow f,$ then
$\left\|  f_{n}\right\|  _{X}\uparrow\left\|  f\right\|  _{X}$ (i.e. Fatou's
Lemma holds in the $X$ norm).}. We shall say that $X$ is a
\textbf{rearrangement-invariant} (r.i.) space, if $g\in X$ implies that all
$\mu-$measurable functions $f$ with $\left|  f\right|  _{\mu}^{\ast}=\left|
g\right|  _{\mu}^{\ast}$ also belong to $X$ and, moreover, $\Vert f\Vert
_{X}=\Vert g\Vert_{X}$. For any r.i. space $X({\Omega})$ we have%

\[
L^{\infty}(\Omega)\subset X(\Omega)\subset L^{1}(\Omega),
\]
with continuous embeddings. Typical examples of r.i. spaces are the $L^{p}%
$-spaces, Orlicz spaces, Lorentz spaces, Marcinkiewicz spaces, etc.

The associated space $X^{\prime}(\Omega)$ is the r.i. space defined by the
following norm%
\[
\left\|  h\right\|  _{X^{\prime}(\Omega)}=\sup_{g\neq0}\frac{\int_{\Omega
}\left|  g(x)h(x)\right|  d\mu}{\left\|  g\right\|  _{X(\Omega)}}=\sup
_{g\neq0}\frac{\int_{0}^{\mu(\Omega)}\left|  g\right|  _{\mu}^{\ast}(s)\left|
h\right|  _{\mu}^{\ast}(s)ds}{\left\|  g\right\|  _{X(\Omega)}}.
\]
In particular, the following \textbf{generalized H\"{o}lder's inequality }holds%

\[
\int_{\Omega}\left|  g(x)h(x)\right|  d\mu\leq\left\|  g\right\|  _{X(\Omega
)}\left\|  h\right\|  _{X^{\prime}(\Omega)}.
\]

Let $X({\Omega})$ be a r.i. space, then there exists a \textbf{unique} r.i.
space (the \textbf{representation space} of $X({\Omega})),$ $\bar{X}=\bar
{X}(0,1)$ on $\left(  \left(  0,1\right)  ,m\right)  $, (where $m$ denotes the
Lebesgue measure on the interval $(0,1)$) such that%
\[
\Vert f\Vert_{X({\Omega})}=\Vert\left|  f\right|  _{\mu}^{\ast}\Vert_{\bar
{X}(0,1)};
\]
and
\[
X^{\prime}(\Omega)=\bar{X}^{\prime}(0,1).
\]
For example, for $1\leq p<\infty,$%
\[
\Vert f\Vert_{L^{p}({\Omega})}=\left(  \int_{{\Omega}}\left|  f\right|
^{p}(x)d\mu\right)  ^{1/p}=\left(  \int_{{0}}^{1}\left(  \left|  f\right|
_{\mu}^{\ast}(s)\right)  ^{p}ds\right)  ^{1/p}=\Vert\left|  f\right|  _{\mu
}^{\ast}\Vert_{\bar{L}^{p}(0,1)},
\]
and%
\[
\Vert f\Vert_{L^{\infty}({\Omega})}=ess\sup\left|  f(z)\right|  =\left|
f\right|  _{\mu}^{\ast}(0^{+})=\Vert\left|  f\right|  _{\mu}^{\ast\ast}%
\Vert_{\bar{L}^{\infty}(0,1)}.
\]
If $\bar{Y}(0,1)$ is a r.i. space $\left(  \left(  0,\mu(\Omega)\right)
,m\right)  ,$ then defining
\[
\Vert f\Vert_{Y({\Omega})}:=\Vert\left|  f_{\mu}\right|  ^{\ast}\Vert_{\bar
{Y}(0,1)}%
\]
we obtain a r.i. space on $({\Omega},\mu),$ in fact there is a one-to-one
correspondence between r.i. spaces on $({\Omega},\mu)$ and those over $\left(
\left(  0,1\right)  ,m\right)  .$ In what follows if there is no possible
confusion we shall use $X$ or $\bar{X}$ without warning.

The following majorization principle holds for r.i. spaces: if
\begin{equation}
\int_{0}^{r}\left|  f\right|  _{\mu}^{\ast}(s)ds\leq\int_{0}^{r}\left|
g\right|  _{\mu}^{\ast}(s)ds, \label{hardy}%
\end{equation}
\ holds for all\ $r>0,$ then, for any r.i. space $\bar{X},$%
\[
\left\|  \left|  f\right|  _{\mu}^{\ast}\right\|  _{\bar{X}}\leq\left\|
\left|  g\right|  _{\mu}^{\ast}\right\|  _{\bar{X}}.
\]

\begin{remark}
\label{Hardy1}The following variant of the majorization principle holds.
Suppose that (\ref{hardy}) holds, then for all $t>0,$%
\[
\Vert\left|  f\right|  _{\mu}^{\ast}(\cdot)\chi_{\lbrack0,t)}(\cdot
)\Vert_{\bar{X}}\leq\Vert\left|  g\right|  _{\mu}^{\ast}(\cdot)\chi
_{\lbrack0,t)}(\cdot)\Vert_{\bar{X}}.
\]
In fact,
\[
\int_{0}^{r}\left|  f\right|  _{\mu}^{\ast}(s)\chi_{\lbrack0,t)}(s)ds=\int
_{0}^{\min\{t,r\}}\left|  f\right|  _{\mu}^{\ast}(s)\leq\int_{0}^{\min
\{t,r\}}\left|  g\right|  _{\mu}^{\ast}(s)ds=\int_{0}^{r}\left|  g\right|
_{\mu}^{\ast}(s)\chi_{\lbrack0,t)}(s)ds,
\]
and we conclude using the majorization principle above.
\end{remark}

The \textbf{fundamental function\ }of $\bar{X}$ is defined by
\[
\phi_{\bar{X}}(s)=\left\|  \chi_{\left[  0,s\right]  }\right\|  _{\bar{X}%
},\text{ \ }0\leq s\leq1,
\]
We can assume without loss of generality that $\phi_{\bar{X}}$ is concave (cf.
\cite{BS}). Moreover, for all $s\in(0,1)$ we have%
\[
\phi_{\bar{X}^{\prime}}(s)\phi_{\bar{X}}(s)=s.
\]
For example, if $1\leq p<\infty,1\leq q\leq\infty,$ and we let $\bar{X}=L^{p}$
or $\bar{X}=L^{p,q}$ (Lorentz space), then $\phi_{L^{p}}(t)=\phi_{L^{p,q}%
}(t)=t^{1/p},$ moreover, $\phi_{L^{\infty}}(t)\equiv1.$ If $N$ is a Young's
function, then the fundamental function of the Orlicz space $\bar{X}=L_{N}$ is
given by $\phi_{L_{N}}(t)=1/N^{-1}(1/t).$

The Lorentz $\Lambda(\bar{X})$ space and the Marcinkiewicz space $M(\bar{X})$
associated with $\bar{X}$ are defined by the quasi-norms
\[
\left\|  f\right\|  _{M(\bar{X})}=\sup_{t}\left|  f\right|  _{\mu}^{\ast
}(t)\phi_{\bar{X}}(t),\text{ \ \ \ \ \ }\left\|  f\right\|  _{\Lambda(\bar
{X})}=\int_{0}^{1}\left|  f\right|  _{\mu}^{\ast}(t)d\phi_{\bar{X}}(t).
\]
Notice that
\[
\phi_{M(\bar{X})}(t)=\phi_{\Lambda(\bar{X})}(t)=\phi_{\bar{X}}(t),
\]
and, moreover,%
\[
\Lambda(\bar{X})\subset\bar{X}\subset M(\bar{X}).\label{tango}%
\]

\subsection{Boyd indices and extrapolation spaces}

The Hardy operators are defined by\footnote{where if $a=0$ we simply let $Q:=$
$Q_{0}$}
\[
Pf(t)=\frac{1}{t}\int_{0}^{t}f(s)ds;\text{ \ \ \ }Q_{a}f(t)=\frac{1}{t^{a}%
}\int_{t}^{1}s^{a}f(s)\frac{ds}{s}\text{ \ (}0\leq a<1).
\]
The boundedness of the Hardy operators on a r.i. space $\bar{X}$ can be
formulated in terms of conditions on the so called \textbf{Boyd indices}%
\footnote{Introduced by D.W. Boyd in \cite{boyd}.}
\[
\bar{\alpha}_{\bar{X}}=\inf\limits_{r>1}\dfrac{\ln h_{\bar{X}}(r)}{\ln
r}\text{ \ \ and \ \ }\underline{\alpha}_{\bar{X}}=\sup\limits_{r<1}\dfrac{\ln
h_{\bar{X}}(r)}{\ln r},
\]
where $h_{\bar{X}}(r)$ denotes the norm of the compression/dilation operator
$E_{s}$ on $\bar{X}$, defined for $s>0,$ by
\begin{equation}
E_{r}f(t)=\left\{
\begin{array}
[c]{ll}%
f^{\ast}(\frac{t}{r}) & 0<t<r,\\
0 & r\leq t.
\end{array}
\right.  \label{verarriba}%
\end{equation}
The operator $E_{s}$ is bounded on every r.i. space $\bar{X},$ moreover,
\[
h_{X}(r)\leq\max\{1,r\},\text{ for all }s>0.
\]
For example, if $\bar{X}=L^{p}$, then $\overline{\alpha}_{L^{p}}%
=\underline{\alpha}_{L^{p}}=\frac{1}{p}.$

We have the following well known fact (cf. \cite{boyd}):
\begin{equation}%
\begin{array}
[c]{c}%
P\text{ is bounded on }\bar{X}\text{ }\Leftrightarrow\overline{\alpha}%
_{\bar{X}}<1,\\
Q_{a}\text{ is bounded on }\bar{X}\text{ }\Leftrightarrow\underline{\alpha
}_{\bar{X}}>a.
\end{array}
\label{alcance}%
\end{equation}
Moreover,
\begin{equation}
\left\|  Q_{a}\right\|  :=\left\|  Q_{a}\right\|  _{\bar{X}\rightarrow\bar{X}%
}\leq\int_{1}^{\infty}h_{\bar{X}}(\frac{1}{s})s^{\frac{1}{a}-1}ds<\infty.
\label{norma}%
\end{equation}

The following extrapolation spaces, introduced by Fiorenza \cite{fio} and
Fiorenza and Karadzhov \cite{fioka} in the special case of $L^{p}$
spaces\footnote{For a discussion of the extrapolation properties of a more
general class of spaces we refer to \cite{astly}.}, will play an important
role in this paper.

\begin{definition}
\label{deffiorenz}Let $\bar{X}$ be a r.i. space, and $k\in\mathbb{N}$. We let
$\bar{X}_{k,\log}$ be the r.i. space defined by%
\[
\bar{X}_{k,\log}=\left\{  f:\left\|  f\right\|  _{\bar{X}_{k,\log}}:=\int
_{0}^{1}\left\|  \left|  f\right|  ^{\ast}(s)\chi_{\lbrack0,t)}(s)\right\|
_{\bar{X}}\frac{dt}{t\left(  \ln\frac{1}{t}\right)  ^{1-k/2}}<\infty\right\}
.
\]
\end{definition}

It can be easily verified that%
\[
\left\|  f\right\|  _{\bar{X}_{k,\log}}\approx\int_{0}^{1}\left\|  \left|
f\right|  ^{\ast}(s)\chi_{\lbrack0,t)}(s)\right\|  _{\bar{X}}\frac
{dt}{t\left(  1+\ln\frac{1}{t}\right)  ^{1-k/2}}.
\]

We now briefly indicate how the $\bar{X}_{k,\log}$ spaces can be identified
with real interpolation/extrapolation spaces of the form (cf. \cite{bl})%
\[
(\bar{X},L^{\infty})_{w_{k},1}=\left\{  f:\left\|  f\right\|  _{(\bar
{X},L^{\infty})_{w_{k},1}}=\int_{0}^{1}K(t,f;\bar{X},L^{\infty})w_{k}%
(t)dt<\infty\right\}  ,
\]
where the $K-$functional, $K(t,f;\bar{X},L^{\infty}),$ is defined (cf.
\cite{bl}) by%
\[
K(t,f;\bar{X},L^{\infty})=\inf_{f=f_{0}+f_{1}}\{\left\|  f_{0}\right\|
_{\bar{X}}+t\left\|  f_{1}\right\|  _{L^{\infty}}\},
\]
and%
\[
w_{k}(t)=\frac{\left(  \phi_{\bar{X}}^{-1}(t)\right)  ^{\prime}}{\phi_{\bar
{X}}^{-1}(t)\left(  1+\ln(\frac{1}{\phi_{\bar{X}}^{-1}(t)})\right)  ^{1-k/2}%
}.
\]
This identification follows readily from the well known formula
(cf.\cite{mifosil})%
\[
K(t,f;\bar{X},L^{\infty})\approx\left\|  \left|  f\right|  ^{\ast}%
\chi_{(0,\phi_{\bar{X}}^{-1}(t))}\right\|  _{\bar{X}}.
\]
For example, if $\bar{X}=L^{p},$ and $k=1,$ then%
\[
(L^{p},L^{\infty})_{w_{1},1}=L_{(p,p^{\prime}}.
\]
This characterization simplifies a number of calculations with these spaces.

\begin{proposition}
\label{indice}Let $\bar{X}$ be a r.i. space on $\Omega$, and $k\in\mathbb{N}$. Then,

(i)%
\begin{equation}
\bar{X}_{k+1,\log}\subset\bar{X}_{k,\log}\subset\bar{X}. \label{inclu}%
\end{equation}

(ii) If $\overline{\alpha}_{\bar{X}}<1,$ then $\overline{\alpha}_{\bar
{X}_{k,\log}}<1.$

(ii) If for some $r>0,$ we have $\underline{\alpha}_{\bar{X}}>r\Rightarrow
\underline{\alpha}_{\bar{X}_{k,\log}}>r.$
\end{proposition}

\begin{proof}
(i) The first inclusion is obvious. To prove the second inclusion we observe
that the identity operator maps%
\[
I:\bar{X}\rightarrow\bar{X},\text{ and }I:L^{\infty}\rightarrow\bar{X},
\]
thus, by interpolation,%
\[
I:\bar{X}_{k,\log}\rightarrow(\bar{X},\bar{X})_{w_{k},1}=\bar{X}.
\]

(ii) By (\ref{alcance}) we need to prove that if $\overline{\alpha}_{\bar{X}%
}<1,$ then $P:\bar{X}_{k,\log}\rightarrow\bar{X}_{k,\log}.$ But $P$ is bounded
on $L^{\infty},$ consequently the result follows interpolating the estimates%
\[
P:\bar{X}\rightarrow\bar{X},\text{ and }P:L^{\infty}\rightarrow L^{\infty}.
\]
(iii) The proof will be by direct estimation of the norm of the
compression/dilation operator $E_{r}f$ (cf. (\ref{verarriba}) above). Let
$0<r<1,$ then
\begin{align*}
\left\|  E_{r}f\right\|  _{\bar{X}_{k,\log}}  &  \leq c\int_{0}^{1}\left\|
\left|  f\right|  ^{\ast}\left(  \frac{s}{r}\right)  \chi_{\lbrack0,r)}%
(s)\chi_{\lbrack0,t)}(s)\right\|  _{\bar{X}}\frac{dt}{t\left(  1+\ln\frac
{1}{t}\right)  ^{1-k/2}}\\
&  =c\int_{0}^{1}\left\|  \left|  f\right|  ^{\ast}\left(  \frac{s}{r}\right)
\chi_{\lbrack0,\min(t,r))}(s)\right\|  _{\bar{X}}\frac{dt}{t\left(  1+\ln
\frac{1}{t}\right)  ^{1-k/2}}\\
&  \leq c\left(  \int_{0}^{r}\left\|  \left|  f\right|  ^{\ast}\left(
\frac{s}{r}\right)  \chi_{\lbrack0,t)}(s)\right\|  _{\bar{X}}\frac
{dt}{t\left(  1+\ln\frac{1}{t}\right)  ^{1-k/2}}+\int_{r}^{1}\left\|  \left|
f\right|  ^{\ast}\left(  \frac{s}{r}\right)  \chi_{\lbrack0,r)}(s)\right\|
_{\bar{X}}\frac{dt}{t\left(  1+\ln\frac{1}{t}\right)  ^{1-k/2}}\right) \\
&  =c(A(r)+B(r)).
\end{align*}
We estimate each of these terms as follows
\begin{align}
A(r)  &  =\int_{0}^{r}\left\|  \left|  f\right|  ^{\ast}\left(  \frac{s}%
{r}\right)  \chi_{\lbrack0,t/r)}(\frac{s}{r})\right\|  _{\bar{X}}\frac
{dt}{t\left(  1+\ln\frac{1}{t}\right)  ^{1-k/2}}\label{l1}\\
&  \leq h_{\bar{X}}(r)\int_{0}^{r}\left\|  \left|  f\right|  ^{\ast}%
(s)\chi_{\lbrack0,t/r)}(s)\right\|  _{\bar{X}}\frac{dt}{t\left(  1+\ln\frac
{1}{t}\right)  ^{1-k/2}}\nonumber\\
&  =h_{\bar{X}}(r)\int_{0}^{1}\left\|  \left|  f\right|  ^{\ast}%
(s)\chi_{\lbrack0,u)}(s)\right\|  _{\bar{X}}\frac{du}{u\left(  1+\ln\frac
{1}{ur}\right)  ^{1-k/2}}\nonumber\\
&  \leq h_{\bar{X}}(r)\sup_{0<u<1}\left(  \frac{1+\ln\frac{1}{u}}{1+\ln
\frac{1}{ur}}\right)  ^{1-k/2}\int_{0}^{1}\left\|  \left|  f\right|  ^{\ast
}(s)\chi_{\lbrack0,u)}(s)\right\|  _{\bar{X}}\frac{du}{u\left(  1+\ln\frac
{1}{u}\right)  ^{1-k/2}}\nonumber\\
&  =h_{\bar{X}}(r)\sup_{0<u<1}\left(  \frac{1+\ln\frac{1}{u}}{1+\ln\frac
{1}{ur}}\right)  ^{1-k/2}\left\|  f\right\|  _{\bar{X}_{k,\log}}.\nonumber
\end{align}
Now, the term containing the supremum can be easily computed. Indeed, by
direct differentiation one sees that the function $\frac{1+\ln\frac{1}{u}%
}{1+\ln\frac{1}{ur}}$ is decreasing, therefore $\left(  \frac{1+\ln\frac{1}%
{u}}{1+\ln\frac{1}{ur}}\right)  ^{1-k/2}$ decreases (resp. increases) when
$1-k/2>0$ (resp. $1-k/2\leq0).$ It follows that
\begin{equation}
\sup_{0<u<1}\left(  \frac{1+\ln\frac{1}{u}}{1+\ln\frac{1}{ur}}\right)
^{1-k/2}=\left\{
\begin{array}
[c]{ll}%
1 & \text{if }k=1,2;\\
\left(  1+\ln\frac{1}{r}\right)  ^{k/2-1} & \text{if }k\geq3.
\end{array}
\right.  \label{l2}%
\end{equation}
We estimate $B(r):$%
\begin{align}
B(r)  &  \leq\left\|  E_{r}f\right\|  _{\bar{X}}\int_{r}^{1}\frac{dt}{t\left(
1+\ln\frac{1}{t}\right)  ^{1-k/2}}\label{l3}\\
&  \leq\left\|  E_{r}f\right\|  _{\bar{X}}\frac{2}{k}\left(  1+\ln\frac{1}%
{r}\right)  ^{k/2}\nonumber\\
&  \leq h_{\bar{X}}(r)\frac{2}{k}\left(  1+\ln\frac{1}{r}\right)
^{k/2}\left\|  f\right\|  _{\bar{X}}\nonumber\\
&  \leq\bar{c}(k)\left(  1+\ln\frac{1}{r}\right)  ^{k/2}\left\|  f\right\|
_{\bar{X}_{k,\log}}\text{ \ (by (\ref{inclu})).}\nonumber
\end{align}
Combining (\ref{l1}), (\ref{l2}) and (\ref{l3}), we see that there exists a
constant $c=c(k),$ such that%
\[
\left\|  E_{r}f\right\|  _{\bar{X}_{k,\log}}\leq ch_{\bar{X}}(r)\left(
1+\ln\frac{1}{r}\right)  ^{k/2}\left\|  f\right\|  _{\bar{X}_{k,\log}}\text{.}%
\]
Therefore,%
\[
h_{\bar{X}_{k,\log}}(r)\leq ch_{\bar{X}}(r)\left(  1+\ln\frac{1}{r}\right)
^{k/2}.
\]
Thus, for $0<r<1,$ we have%
\[
\dfrac{\ln h_{\bar{X}_{k,\log}}(r)}{\ln r}\geq\frac{\ln h_{\bar{X}}(r)}{\ln
r}+\frac{\ln c\left(  1+\ln\frac{1}{r}\right)  ^{k/2}}{\ln r}.
\]
It follows that%
\begin{align*}
\underline{\alpha}_{\bar{X}_{k,\log}}  &  =\sup_{0<r<1}\dfrac{\ln h_{\bar
{X}_{k,\log}}(r)}{\ln r}\\
&  =\lim_{r\mapsto0}\dfrac{\ln h_{\bar{X}_{k,\log}}(r)}{\ln r}\\
&  \geq\lim_{r\mapsto0}\left\{  \frac{\ln h_{\bar{X}}(r)}{\ln r}+\frac{\ln
c\left(  1+\ln\frac{1}{r}\right)  ^{k/2}}{\ln r}\right\} \\
&  =\underline{\alpha}_{\bar{X}}+\lim_{r\rightarrow0}\frac{\ln c\left(
1+\ln\frac{1}{r}\right)  ^{k/2}}{\ln r}\\
&  =\underline{\alpha}_{\bar{X}},
\end{align*}

as we wished to show.
\end{proof}

\section{The main Theorem\label{thema}}

In this section we always work with $\left(  \Omega,d,\mu\right)  $
probability metric spaces, as described in the previous section, and will
always let $J$ denote an isoperimetric estimator of $\left(  \Omega
,d,\mu\right)  .$

\begin{theorem}
\label{main}Let $\bar{X}$ be a rearrangement invariant space$,$ and suppose
that $G:(0,1/2)\rightarrow(0,\infty),$ satisfies
\[
\int_{0}^{\frac{1}{2}}\frac{t}{J(t)}G(t)dt<\infty.
\]
Then:

\begin{enumerate}
\item  If $Q$ is bounded on $\bar{X},$ then the following Sobolev inequality
holds: $\forall f\in Lip(\Omega),$%
\begin{equation}
\int_{0}^{\frac{1}{2}}\left\|  \left(  f^{\ast}(\cdot)-f^{\ast}(t)\right)
\chi_{\lbrack0,t)}(\cdot)\right\|  _{\bar{X}}G(t)dt\leq\left\|  Q\right\|
_{\bar{X}\rightarrow\bar{X}}\left\|  \left|  \nabla f\right|  ^{\ast}\right\|
_{\bar{X}}\int_{0}^{\frac{1}{2}}\frac{t}{J(t)}G(t)dt. \label{aqaq}%
\end{equation}

\item  If $P$ is bounded on $\bar{X},$ then the following Sobolev inequality
holds: $\forall f\in Lip(\Omega),$%
\begin{equation}
\int_{0}^{\frac{1}{2}}\left\|  (f^{\ast\ast}(\cdot)-f^{\ast}(\cdot
))\chi_{\lbrack0,t)}(\cdot)\right\|  _{\bar{X}}G(t)dt\leq\left\|  P\right\|
_{\bar{X}\rightarrow\bar{X}}\left\|  \left|  \nabla f\right|  ^{\ast}\right\|
_{\bar{X}}\int_{0}^{\frac{1}{2}}\frac{t}{J(t)}G(t)dt. \label{aqaq1}%
\end{equation}
\end{enumerate}
\end{theorem}

\begin{proof}
Part 2: To complete the details of the proof outlined in the introduction,
simply note that the inequality (\ref{larusa}) above, follows directly from
(\ref{reod00}). The proof of first \ part of the theorem requires an extra
argument. Suppose that $f\in Lip(\Omega),$ then $f^{\ast}$ is locally
absolutely continuous and, since $f^{\ast}$ is decreasing, it follows that
$\left(  -f^{\ast}\right)  ^{^{\prime}}\geq0.$ By the fundamental theorem of
calculus we can write
\[
f^{\ast}(s)-f^{\ast}(t)=\int_{s}^{t}\left(  -f^{\ast}\right)  ^{^{\prime}%
}(z)dz,\text{ \ \ }0<s<t<\frac{1}{2}.
\]
Consequently,
\begin{align*}
\left\|  \left(  f^{\ast}(s)-f^{\ast}(t)\right)  \chi_{\lbrack0,t)}%
(s)\right\|  _{\bar{X}}  &  =\left\|  \int_{s}^{t}\left(  -f^{\ast}\right)
^{^{\prime}}(z)dz\right\|  _{\bar{X}}\\
&  =\left\|  \left(  \int_{s}^{1}J(z)\frac{z}{J(z)}\left(  -f^{\ast}\right)
^{^{\prime}}(z)\chi_{\lbrack0,t)}(z)\frac{dz}{z}\right)  \right\|  _{\bar{X}%
}\\
&  \leq\left\|  Q\right\|  _{\bar{X}\rightarrow\bar{X}}\left\|  \left(
\frac{z}{J(z)}\right)  J(z)\left(  -f^{\ast}\right)  ^{^{\prime}}%
(z)\chi_{\lbrack0,t)}(z)\right\|  _{\bar{X}}\text{ \ \ (since }Q\text{ is
bounded on }\bar{X})\\
&  =\left\|  Q\right\|  _{\bar{X}\rightarrow\bar{X}}\left(  \frac{t}%
{J(t)}\right)  \left\|  \left(  J(\cdot)\left(  -f^{\ast}\right)  ^{^{\prime}%
}(\cdot)\chi_{\lbrack0,t)}(\cdot)\right)  ^{\ast}\right\|  _{\bar{X}}\text{
\ (since }\frac{z}{J(z)}\uparrow\text{)}\\
&  \leq\left\|  Q\right\|  _{\bar{X}\rightarrow\bar{X}}\frac{t}{J(t)}\left\|
\left|  \nabla f\right|  ^{\ast}(z)\chi_{\lbrack0,t]}(z)\right\|  _{\bar{X}%
}\text{ \ \ (by (\ref{aa}) and Remark \ref{Hardy1})}\\
&  \leq\left\|  Q\right\|  _{\bar{X}\rightarrow\bar{X}}\frac{t}{J(t)}\left\|
\left|  \nabla f\right|  ^{\ast}\right\|  _{\bar{X}}.
\end{align*}
Thus,%
\[
\int_{0}^{\frac{1}{2}}\left\|  \left(  f^{\ast}(\cdot)-f^{\ast}(t)\right)
\chi_{\lbrack0,t)}(\cdot)\right\|  _{\bar{X}}G(t)dt\leq\left\|  Q\right\|
_{\bar{X}\rightarrow\bar{X}}\left\|  \left|  \nabla f\right|  ^{\ast}\right\|
_{\bar{X}}\int_{0}^{\frac{1}{2}}\frac{t}{J(t)}G(t)dt.
\]
\end{proof}

\begin{remark}
Since $f\in Lip(\Omega)\Rightarrow$ $\left|  f\right|  \in Lip(\Omega)$ with
$\left|  \nabla\left|  f\right|  \right|  \leq\left|  \nabla f\right|  ,$ the
inequalities (\ref{aqaq}) and (\ref{aqaq1}) also hold for $\left|  f\right|  .$
\end{remark}

Let us also note the following converse to Theorem \ref{main}

\begin{corollary}
\label{ma1}Let $r\in(0,1]$ and suppose that suppose that
\begin{equation}
\int_{0}^{r}\left\|  \left(  f^{\ast}(\cdot)-f^{\ast}(t)\right)  \chi
_{\lbrack0,t)}(\cdot)\right\|  _{\bar{X}}G(t)dt\leq C(X)\left\|  \left|
\nabla f\right|  ^{\ast}\right\|  _{\bar{X}},\label{habra}%
\end{equation}
holds for all r.i. spaces $X$ away from $L^{\infty}.$ Then, for all Borel sets
$A\subset\Omega,$ with $\mu(A)\leq r,$ we have
\begin{equation}
\mu(A)\int_{\mu(A)}^{r}G(t)dt\leq CP(A;\Omega),\label{estima}%
\end{equation}
and consequently,%
\[
\frac{t}{I(t)}\int_{t}^{r}G(t)dt\leq C,\text{ for all }t\in(0,r).
\]
\end{corollary}

\begin{proof}
Our assumption implies that the inequality (\ref{habra}) holds for $X=L^{1}.$
Let $A$ be a Borel set with $\mu(A)\leq r$. We may assume without loss of
generality that $P(A;\Omega)<\infty.$ By \cite{bobk} we can select a sequence
$\{f_{n}\}_{n\in N}$ of Lip functions such that $f_{n}\underset{L^{1}%
}{\rightarrow}\chi_{A}$, and%
\[
P(A;\Omega)=\lim\sup_{n\rightarrow\infty}\left\|  \left|  \nabla f_{n}\right|
\right\|  _{L^{1}}.
\]
Therefore, by (\ref{habra}) applied to the sequence of $f_{n}^{\prime}s$
above, we obtain
\[
\lim\sup_{n\rightarrow\infty}\int_{0}^{r}\left(  \int_{0}^{t}\left(  \left(
f_{n}\right)  _{\mu}^{\ast}(s)-\left(  f_{n}\right)  _{\mu}^{\ast}(t)\right)
ds\right)  G(t)dt\leq CP(A;\Omega).
\]
It is known that $f_{n}\underset{L^{1}}{\rightarrow}\chi_{A}$ implies that
(cf. \cite[Lemma 2.1]{gar}):
\[
\left(  f_{n}\right)  _{\mu}^{\ast}(t)\rightarrow\left(  \chi_{A}\right)
_{\mu}^{\ast}(t)=\chi_{\lbrack0,\mu(A)]}(t)\text{ }\text{at all points of
continuity of }\left(  \chi_{A}\right)  _{\mu}^{\ast}.
\]
Consequently,%
\begin{align*}
\mu(A)\int_{\mu(A)}^{r}G(t)dt &  \leq\lim\sup_{n\rightarrow\infty}\int_{0}%
^{r}\left(  \int_{0}^{t}\left(  \left(  f_{n}\right)  _{\mu}^{\ast}(s)-\left(
f_{n}\right)  _{\mu}^{\ast}(t)\right)  ds\right)  G(t)dt\\
&  \leq CP(A;\Omega).
\end{align*}
\end{proof}

\begin{definition}
Let $J$ be an isoperimetric estimator of $\left(  \Omega,d,\mu\right)  $.\ The
\textbf{isoperimetric Hardy operator} $Q_{J}$ is defined by
\[
Q_{J}f(t):=\frac{J(t)}{t}\int_{t}^{\frac{1}{2}}f(z)\frac{dz}{J(z)}.
\]
\end{definition}

\begin{theorem}
\label{inclusion}Let $\bar{X}$ be a rearrangement invariant space$,$ and let
$J$ be an isoperimetric estimator. Let $G:(0,1)\rightarrow(0,\infty)$ be such
that
\[
\int_{0}^{\frac{1}{2}}\frac{t}{J(t)}G(t)dt<\infty.
\]
Suppose that the isoperimetric operator $Q_{J}$ is bounded on $\bar{X}.$ Let
$f\in Lip(\Omega)$ and let $med(f)$ be a median\footnote{Let $f\ $be a
measurable function, a real number $med(f)$ will be called a \textbf{median}
of $f$ if
\[
\mu\left\{  f\geq med(f)\right\}  \geq1/2\text{ \ and }\mu\left\{  f\leq
med(f)\right\}  \geq1/2.
\]
} of $f,$ then%
\[
\int_{0}^{\frac{1}{2}}\left\|  (f-med(f))^{\ast}(s)\chi_{\lbrack
0,t)}(s)\right\|  _{\bar{X}}G(t)dt\leq\left(  \left\|  Q\right\|  _{\bar
{X}\rightarrow\bar{X}}+\left\|  Q_{J}\right\|  _{\bar{X}\rightarrow\bar{X}%
}\right)  \left\|  \left|  \nabla f\right|  ^{\ast}\right\|  _{\bar{X}}%
\int_{0}^{\frac{1}{2}}\frac{t}{J(t)}G(t)dt.
\]
\end{theorem}

\begin{proof}
Let us start by remarking that since $\frac{t}{J(t)}$ is increasing, for
$f\geq0$ we have
\[
Q_{J}f(t)=\frac{J(t)}{t}\int_{t}^{\frac{1}{2}}f(z)\frac{dz}{J(z)}\geq\int
_{t}^{\frac{1}{2}}f(z)\frac{dz}{z}=Qf(t).
\]
Consequently, if $Q_{J}$ is bounded on $\bar{X}$ then $Q$ is also bounded on
$\bar{X}.$

Let $f\in Lip(\Omega)$, and let $0<s<t<\frac{1}{2}.$ Since $f^{\ast}$ is
decreasing, we have
\begin{align*}
\left\|  \left(  f^{\ast}(s)-f^{\ast}(1/2)\right)  \chi_{\lbrack
0,t)}(s)\right\|  _{\bar{X}}  &  \leq\left\|  (f^{\ast}(s)-f^{\ast}%
(t))\chi_{\lbrack0,t)}(s)\right\|  _{\bar{X}}+\left|  f^{\ast}(t)-f^{\ast
}(1/2)\right|  \left\|  \chi_{\lbrack0,t)}(s)\right\|  _{\bar{X}}\\
&  =\left\|  (f^{\ast}(s)-f^{\ast}(t))\chi_{\lbrack0,t)}(s)\right\|  _{\bar
{X}}+\left(  f^{\ast}(t)-f^{\ast}(1/2)\right)  \left\|  \chi_{\lbrack
0,t)}(s)\right\|  _{\bar{X}}\\
&  =(A)+(B).
\end{align*}
By the proof of Theorem \ref{main} we know that
\begin{equation}
(A)\leq\frac{t}{J(t)}\left\|  Q\right\|  _{\bar{X}\rightarrow\bar{X}}\left\|
\left|  \nabla f\right|  ^{\ast}\right\|  _{\bar{X}}. \label{aa1}%
\end{equation}
We estimate the second term as follows:
\begin{align*}
(B)  &  =\left(  \int_{t}^{\frac{1}{2}}\left(  -f^{\ast}\right)  ^{^{\prime}%
}(z)dz\right)  \phi_{\bar{X}}(t)\\
&  =\frac{t}{J(t)}\left(  \frac{J(t)}{t}\int_{t}^{\frac{1}{2}}\left(
J(z)\left(  -f^{\ast}\right)  ^{^{\prime}}(z)\right)  \frac{dz}{J(z)}\right)
\phi_{\bar{X}}(t)\\
&  =\frac{t}{J(t)}Q_{J}\left(  \left(  J(z)\left(  -f^{\ast}\right)
^{^{\prime}}(z)\right)  \right)  (t)\phi_{\bar{X}}(t)\\
&  \leq\frac{t}{J(t)}\sup_{t}\left[  Q_{J}\left(  J(\cdot)\left(  -f^{\ast
}\right)  ^{^{\prime}}(\cdot)\right)  (t)\phi_{\bar{X}}(t)\right] \\
&  =\frac{t}{J(t)}\left\|  Q_{J}\left(  J(\cdot)\left(  -f^{\ast}\right)
^{^{\prime}}(\cdot)\right)  \right\|  _{M(\bar{X})}.
\end{align*}
Thus,
\begin{align*}
\left(  f^{\ast}(t)-f^{\ast}(1/2)\right)  \phi_{\bar{X}}(t)  &  \leq\frac
{t}{J(t)}\sup_{t}\left[  Q_{J}\left(  J(\cdot)\left(  -f^{\ast}\right)
^{^{\prime}}(\cdot)\right)  (t)\phi_{\bar{X}}(t)\right] \\
&  =\frac{t}{J(t)}\left\|  Q_{J}\left(  J(\cdot)\left(  -f^{\ast}\right)
^{^{\prime}}(\cdot)\right)  \right\|  _{M(\bar{X})}\\
&  \leq\frac{t}{J(t)}\left\|  Q_{J}\left(  J(\cdot)\left(  -f^{\ast}\right)
^{^{\prime}}(\cdot)\right)  \right\|  _{\bar{X}}.
\end{align*}
Since we are assuming that $Q_{J}$ is bounded on $\bar{X},$ we have
\begin{align}
\left\|  Q_{J}\left(  J(\cdot)\left(  -f^{\ast}\right)  ^{^{\prime}}%
(\cdot)\right)  \right\|  _{\bar{X}}  &  \leq\left\|  Q_{J}\right\|  _{\bar
{X}\rightarrow\bar{X}}\left\|  J(\cdot)\left(  -f^{\ast}\right)  ^{^{\prime}%
}(\cdot)\right\|  _{\bar{X}}\label{aa2}\\
&  \leq\left\|  Q_{J}\right\|  _{\bar{X}\rightarrow\bar{X}}\left\|  \left|
\nabla f\right|  ^{\ast}\right\|  _{\bar{X}}\text{ \ \ (by (\ref{aa}%
))}.\nonumber
\end{align}
Adding the estimates for $(A)$ and $(B)$ (cf. (\ref{aa1}) and (\ref{aa2})
above) we obtain
\[
\left\|  \left(  f^{\ast}(s)-f^{\ast}(1/2)\right)  \chi_{\lbrack
0,t)}(s)\right\|  _{\bar{X}}\leq\frac{t}{J(t)}\left(  \left\|  Q\right\|
_{\bar{X}\rightarrow\bar{X}}+\left\|  Q_{J}\right\|  _{\bar{X}\rightarrow
\bar{X}}\right)  \left\|  \left|  \nabla f\right|  ^{\ast}\right\|  _{\bar{X}%
}.
\]
It is easy to see that $f^{\ast}(\frac{1}{2})$ is a median of $f$ (cf.
\cite{mamiast}), moreover, since for any constant $a,$ we have $f^{\ast
}(s)-a=(f-a)^{\ast}(s),$ we finally arrive at
\[
\int_{0}^{\frac{1}{2}}\left\|  (f-med(f))^{\ast}(s)\chi_{\lbrack
0,t)}(s)\right\|  _{\bar{X}}G(t)dt\leq\left(  \left\|  Q\right\|  _{\bar
{X}\rightarrow\bar{X}}+\left\|  Q_{J}\right\|  _{\bar{X}\rightarrow\bar{X}%
}\right)  \left\|  \left|  \nabla f\right|  ^{\ast}\right\|  _{\bar{X}}%
\int_{0}^{\frac{1}{2}}\frac{t}{J(t)}G(t)dt.
\]
\end{proof}

\section{Applications\label{appl}}

\subsection{Homogeneous Sobolev spaces\label{markao}}

In this subsection we consider bounded domains $\Omega\subset\mathbb{R}^{n}$
normalized so that $\left|  \Omega\right|  =1.$ We consider the Sobolev space
$W_{0}^{k,1}(\Omega)$ of functions $f$ $\in L^{1}(\Omega)$ that are $k-$ times
weakly differentiable on $\Omega$ and such that their continuation by $0$
outside $\Omega$ are $k-$ times weakly differentiable functions on
$\mathbb{R}^{n}.$ For $f\in W_{0}^{k,1}(\Omega)$ we then have $f\in
W_{0}^{k,1}(\mathbb{R}^{n}),$ with%
\[
\left\|  \left|  D^{j}f\right|  \right\|  _{L^{1}(\Omega)}=\left\|  \left|
D^{j}f\right|  \right\|  _{L^{1}(\mathbb{R}^{n})}\text{ \ \ \ }(j=0,1,\cdots
k).
\]
More generally, given $\bar{X}$ a r.i. space on $\left(  0,1\right)  $, the
Sobolev space $W_{0}^{k,\bar{X}}:=W_{0}^{k,\bar{X}}(\Omega),$ will be defined
as%
\[
W_{0}^{k,\bar{X}}=\left\{  f\in W_{0}^{k,1}(\Omega):\left\|  f\right\|
_{W_{0}^{1,\bar{X}}}:=\sum_{j=0}^{k}\left\|  \left|  D^{j}f\right|  ^{\ast
}\right\|  _{\bar{X}}<\infty\right\}  .
\]

Let%
\[
I_{n}(t)=n\left(  \gamma_{n}\right)  ^{1/n}t^{1-1/n},
\]
where $\gamma_{n}=\frac{\pi^{n/2}}{\Gamma(1+n/2)}$ is the measure of the unit
ball in $\mathbb{R}^{n}$ (i.e. $I_{n}(t)$ is the isoperimetric profile
associated to $\mathbb{R}^{n}).$

Let $f\in W_{0}^{1,1}$ then (see \cite{mmp} and \cite{mamiadv}):

\begin{enumerate}
\item
\begin{equation}
f^{\ast\ast}(t)-f^{\ast}(t)\leq\frac{t}{I_{n}(t)}\frac{1}{t}\int_{0}%
^{t}\left|  \nabla f\right|  ^{\ast}(s)ds,\text{ }0<t<1. \label{aatot0}%
\end{equation}

\item $f^{\ast}$ is locally absolutely continuous, and
\begin{equation}
\int_{0}^{t}\left|  (-f^{\ast})^{\prime}(\cdot)I_{n}(\cdot)\right|  ^{\ast
}(s)\leq\int_{0}^{t}\left|  \nabla f\right|  ^{\ast}(s)ds. \label{aatot}%
\end{equation}
\end{enumerate}

Using (\ref{aatot0}) and (\ref{aatot}) and the method of proof of Theorem
\ref{main} we readily obtain

\begin{theorem}
\label{teohomogeneo}Let $\bar{X}$ be a r.i. space. Let $G:(0,1)\rightarrow
(0,\infty)$ be such that
\begin{equation}
\int_{0}^{1}\frac{t}{I_{n}(t)}G(t)dt<\infty. \label{computada}%
\end{equation}
Then,

\begin{enumerate}
\item  If $\left\|  Q\right\|  _{\bar{X}\rightarrow\bar{X}}<\infty,$ then for
all $f\in W_{0}^{1,\bar{X}},$%
\[
\int_{0}^{1}\left\|  \left(  f^{\ast}(\cdot)-f^{\ast}(t)\right)  \chi
_{\lbrack0,t)}(\cdot)\right\|  _{\bar{X}}G(t)dt\leq\left\|  Q\right\|
_{\bar{X}\rightarrow\bar{X}}\left\|  \left|  \nabla f\right|  ^{\ast}\right\|
_{\bar{X}}\int_{0}^{1}\frac{t}{I_{n}(t)}G(t)dt.
\]

\item  If $\left\|  P\right\|  _{\bar{X}\rightarrow\bar{X}}<\infty,$ then for
all $f\in W_{0}^{1,\bar{X}},$%
\[
\int_{0}^{1}\left\|  (f^{\ast\ast}(\cdot)-f^{\ast}(\cdot))\chi_{\lbrack
0,t)}(\cdot)\right\|  _{\bar{X}}G(t)dt\leq\left\|  P\right\|  _{\bar
{X}\rightarrow\bar{X}}\left\|  \left|  \nabla f\right|  ^{\ast}\right\|
_{\bar{X}}\int_{0}^{1}\frac{t}{I_{n}(t)}G(t)dt.
\]
\end{enumerate}
\end{theorem}

In order to describe in detail the consequences of the previous result we need
to compute the integral (\ref{computada}). Towards this end let us consider
the function
\[
G(t)=\frac{1}{t\sqrt{\ln\left(  \frac{1}{t}\right)  }},\text{ }t\in(0,1).
\]
Then,
\begin{align}
\int_{0}^{1}\frac{t}{tI_{n}(t)}G(t)dt  &  =\frac{1}{n\left(  \gamma
_{n}\right)  ^{1/n}}\int_{0}^{1}t^{1/n}\frac{dt}{t\left(  \ln\frac{1}%
{t}\right)  ^{\frac{1}{2}}}\label{cuenta}\\
&  =\frac{1}{n\left(  \gamma_{n}\right)  ^{1/n}}\int_{0}^{\infty}z^{-\frac
{1}{2}}e^{-z/n}dz\nonumber\\
&  =\frac{\sqrt{\pi}n^{\frac{1}{2}}}{n\left(  \gamma_{n}\right)  ^{1/n}%
}\nonumber\\
&  =\frac{\Gamma(1+\frac{n}{2})^{1/n}}{n^{\frac{1}{2}}}.\nonumber
\end{align}
Consequently, we have the following

\begin{corollary}
\label{teo01}Let $\bar{X}$ be a r.i. space on $(0,1)$. Then,

\begin{enumerate}
\item  If $\underline{\alpha}_{X}>0,$ then, for all $f\in W_{0}^{1,\bar{X}}\ $%
\[
\int_{0}^{1}\left\|  \left(  f^{\ast}(s)-f^{\ast}(t)\right)  \chi
_{\lbrack0,t)}(s)\right\|  _{\bar{X}}\frac{dt}{t\left(  \ln\frac{1}{t}\right)
^{\frac{1}{2}}}\leq\frac{\Gamma(1+\frac{n}{2})^{1/n}}{n^{\frac{1}{2}}}\left\|
Q\right\|  _{\bar{X}\rightarrow\bar{X}}\left\|  \left|  \nabla f\right|
^{\ast}\right\|  _{\bar{X}}.
\]

\item  If $\overline{\alpha}_{X}<1,$ then, for all $f\in W_{0}^{1,\bar{X}},$%
\[
\int_{0}^{1}\left\|  \left(  f^{\ast\ast}(s)-f^{\ast}(s)\right)  \chi
_{\lbrack0,t)}(s)\right\|  _{\bar{X}}\frac{dt}{t\left(  \ln\frac{1}{t}\right)
^{\frac{1}{2}}}\leq\frac{\Gamma(1+\frac{n}{2})^{1/n}}{n^{\frac{1}{2}}}\left\|
P\right\|  _{\bar{X}\rightarrow\bar{X}}\left\|  \left|  \nabla f\right|
^{\ast}\right\|  _{\bar{X}}.
\]
\end{enumerate}
\end{corollary}

Is easy to see that Corollary \ref{teo01} gives the main result of \cite{FKS}
as a special case. In fact, we will now show an extension, valid for higher
derivatives, which for easier comparison, we shall formulate in terms of the
spaces defined in Definition \ref{deffiorenz} above.

The isoperimetric operator in this case is given by%
\[
Q_{I_{n}}f(t):=\frac{I_{n}(t)}{t}\int_{t}^{1}f(z)\frac{dz}{I_{n}(z)}%
=t^{-1/n}\int_{t}^{1}z^{1/n}f(z)\frac{dz}{z}.
\]
Observe that $Q_{I_{n}}$ is bounded on $\bar{X}$ if and only if $\underline
{\alpha}_{\bar{X}}>1/n.$

\begin{theorem}
\label{ordenk}Let $\bar{X}$ be a r.i. space such that $\underline{\alpha
}_{\bar{X}}>0$. Let $M$ be\ the smallest natural number such that
\begin{equation}
\underline{\alpha}_{\bar{X}}>1/M. \label{idice}%
\end{equation}
Let $\Omega\subset\mathbb{R}^{n}$ be a bounded domain such that $\left|
\Omega\right|  =1,$ and suppose that $n\geq M.$ Then for all $f\in
W_{0}^{k,\bar{X}},$ we have
\begin{equation}
\left\|  f\right\|  _{\bar{X}_{k,\log}}\leq c(M,k,\bar{X})\left\|  \left|
D^{k}f\right|  ^{\ast}\right\|  _{\bar{X}},\text{ } \label{pfin}%
\end{equation}
where the constant $c(M,k,\bar{X})$ does not depend on the dimension$.$
\end{theorem}

\begin{proof}
We proceed by induction. Let $k=1,$ and let $f\in W_{0}^{1,\bar{X}}$. For
$0<s<t<1,$ we have%
\begin{align*}
\left\|  \left|  f\right|  ^{\ast}(s)\chi_{\lbrack0,t)}(s)\right\|  _{\bar
{X}}  &  \leq\left\|  (\left|  f\right|  ^{\ast}(s)-\left|  f\right|  ^{\ast
}(t))\chi_{\lbrack0,t)}(s)\right\|  _{\bar{X}}+\left|  f\right|  ^{\ast
}(t)\left\|  \chi_{\lbrack0,t)}(s)\right\|  _{\bar{X}}\\
&  =(A)+(B).
\end{align*}
By the proof of Theorem \ref{main} we have
\begin{equation}
(A)\leq\frac{t^{1/n}}{n\gamma_{n}^{1/n}}\left\|  Q\right\|  _{\bar
{X}\rightarrow\bar{X}}\left\|  \left|  \nabla f\right|  ^{\ast}\right\|
_{\bar{X}}. \label{aaa1}%
\end{equation}
Now, since $f\in W_{0}^{1,\bar{X}}$ implies $\left|  f\right|  \in
W_{0}^{1,\bar{X}}$, and moreover, since $\left|  f\right|  ^{\ast}(1)=0,\ $we
can write
\[
\left|  f\right|  ^{\ast}(t)=\int_{t}^{1}\left(  -\left|  f\right|  ^{\ast
}\right)  ^{^{\prime}}(z)dz,\text{ \ \ }0<t<1.
\]
Consequently,%
\[
(B)=\left(  \int_{t}^{1}\left(  -\left|  f\right|  ^{\ast}\right)  ^{^{\prime
}}(z)dz\right)  \phi_{\bar{X}}(t).
\]
From this point we follow the proof of Theorem \ref{inclusion} to obtain%
\[
\left(  \int_{t}^{1}\left(  -\left|  f\right|  ^{\ast}\right)  ^{^{\prime}%
}(z)dz\right)  \phi_{\bar{X}}(t)\leq\frac{t^{1/n}}{n\gamma_{n}^{1/n}}\left\|
Q_{I_{n}}\right\|  _{\bar{X}\rightarrow\bar{X}}\left\|  \left|  \nabla
f\right|  ^{\ast}\right\|  _{\bar{X}}.
\]
Adding the estimates obtained for $(A)$ and $(B)$ we get
\[
\left\|  \left|  f\right|  ^{\ast}(s)\chi_{\lbrack0,t)}(s)\right\|  _{\bar{X}%
}\leq\frac{t^{1/n}}{n\gamma_{n}^{1/n}}\left(  \left\|  Q\right\|  _{\bar
{X}\rightarrow\bar{X}}+\left\|  Q_{I_{n}}\right\|  _{\bar{X}\rightarrow\bar
{X}}\right)  \left\|  \left|  \nabla f\right|  ^{\ast}\right\|  _{\bar{X}}.
\]

Therefore, using (\ref{cuenta}), (\ref{norma}) and (\ref{idice}) we get%
\begin{align*}
\left\|  f\right\|  _{\bar{X}_{1,\log}}  &  =\int_{0}^{1}\left\|  \left|
f\right|  ^{\ast}(s)\chi_{\lbrack0,t)}(s)\right\|  _{\bar{X}}\frac
{dt}{t\left(  \ln\frac{1}{t}\right)  ^{\frac{1}{2}}}\\
&  \leq\frac{\Gamma(1+\frac{n}{2})^{1/n}}{n^{\frac{1}{2}}}\left(  \left\|
Q\right\|  _{\bar{X}\rightarrow\bar{X}}+\int_{1}^{\infty}h_{\bar{X}}(\frac
{1}{s})s^{\frac{1}{M}-1}ds\right)  \left\|  \left|  \nabla f\right|  ^{\ast
}\right\|  _{\bar{X}}.
\end{align*}
Recall that for $x\geq1,$ we have $\Gamma(x)\leq x^{x}$; consequently%
\[
\frac{\Gamma(1+\frac{n}{2})^{1/n}}{n^{\frac{1}{2}}}=\left(  \frac{n}%
{2}\right)  ^{1/n}\frac{\Gamma(\frac{n}{2})^{1/n}}{n^{\frac{1}{2}}}\leq
\frac{1}{\sqrt{2}}\left(  \frac{n}{2}\right)  ^{1/n}\leq c.
\]
Thus,%
\[
\left\|  f\right\|  _{\bar{X}_{1,\log}}\leq c(M,1,\bar{X})\left\|  \left|
\nabla f\right|  \right\|  _{\bar{X}},
\]
where $c(M,1,\bar{X})$ is a constant that does not depend on $n.$

Let $k\geq2,$ and suppose that the desired inequality is valid for $k-1.$ Let
$f\in W_{0}^{k,\bar{X}},$ then, by the induction hypothesis, and the fact that
$\left|  \nabla f\right|  \in W_{0}^{k-1,\bar{X}},$ we have%
\begin{align}
\left\|  \left|  \nabla f\right|  \right\|  _{\bar{X}_{k-1,\log}}  &
:=\int_{0}^{1}\left\|  \left|  \nabla f\right|  ^{\ast}(s)\chi_{\lbrack
0,t)}(s)\right\|  _{\bar{X}}\frac{dt}{t\left(  \ln\frac{1}{t}\right)
^{1-(k-1)/2}}\label{qqq}\\
&  \leq c(M,k-1,\bar{X})\left\|  \left|  D^{k-1}\left|  \nabla f\right|
\right|  ^{\ast}\right\|  _{\bar{X}}.\nonumber
\end{align}
By Proposition \ref{indice} (part 3), the r.i. space $\bar{X}_{k-1,\log}$
satisfies $\underline{\alpha}_{\bar{X}_{k-1,\log}}>r.$ Consequently we may
apply the result obtained in the first step of the proof to the space $\bar
{X}_{k-1,\log}$ , and we obtain%
\begin{align}
\left\|  f\right\|  _{(\bar{X}_{k-1,\log})_{1,\log}}  &  \leq c(M,1,\bar
{X})\left\|  \left|  \nabla f\right|  ^{\ast}\right\|  _{\bar{X}_{k-1,\log}%
}\nonumber\\
&  \leq c(M,1,\bar{X})c(M,k-1,\bar{X})\left\|  \left|  D^{k-1}\left|  \nabla
f\right|  \right|  ^{\ast}\right\|  _{\bar{X}}\text{ \ (by (\ref{qqq}%
))}\nonumber\\
&  \leq c(M,1,\bar{X})c(M,k-1,\bar{X})\left\|  \left|  D^{k}f\right|  ^{\ast
}\right\|  _{\bar{X}}. \label{uff}%
\end{align}
We will show in a moment that
\begin{equation}
\left\|  f\right\|  _{(\bar{X}_{k-1,\log})_{1,\log}}=\frac{2k}{k-1}\left\|
f\right\|  _{\bar{X}_{k,\log}}. \label{concluira}%
\end{equation}
Assuming (\ref{concluira}) and combining it with (\ref{uff}) we see that%
\[
\left\|  f\right\|  _{\bar{X}_{k,\log}}\leq c(M,k,\bar{X})\left\|  \left|
D^{k}f\right|  ^{\ast}\right\|  _{\bar{X}}.
\]
It thus remains to prove (\ref{concluira}). For this purpose we write
\begin{align*}
\left\|  f\right\|  _{(\bar{X}_{k-1,\log})_{1,\log}}  &  =\int_{0}^{1}\left\|
\left|  f\right|  ^{\ast}(s)\chi_{\lbrack0,t)}(s)\right\|  _{\bar{X}%
_{k-1,\log}}\frac{dt}{t\left(  \ln\frac{1}{t}\right)  ^{\frac{1}{2}}}\\
&  =\int_{0}^{1}\left\|  \left(  \left|  f\right|  ^{\ast}\chi_{\lbrack
0,t)}\right)  ^{\ast}(\cdot)\chi_{\lbrack0,s)}(\cdot)\right\|  _{\bar
{X}_{k-1,\log}}\frac{dt}{t\left(  \ln\frac{1}{t}\right)  ^{\frac{1}{2}}}\\
&  =\int_{0}^{1}\left(  \int_{0}^{1}\left\|  \left|  f\right|  ^{\ast}%
(\cdot)\chi_{\lbrack0,t)}(\cdot)\chi_{\lbrack0,s)}(\cdot)\right\|  _{\bar{X}%
}\frac{ds}{s\left(  \ln\frac{1}{s}\right)  ^{1-(k-1)/2}}\right)  \frac
{dt}{t\left(  \ln\frac{1}{t}\right)  ^{\frac{1}{2}}}\\
&  =\int_{0}^{1}\left(  \int_{0}^{1}\left\|  \left|  f\right|  ^{\ast}%
(\cdot)\chi_{\lbrack0,\min(s,t))}(\cdot)\right\|  _{\bar{X}}\frac{ds}{s\left(
\ln\frac{1}{s}\right)  ^{1-(k-1)/2}}\right)  \frac{dt}{t\left(  \ln\frac{1}%
{t}\right)  ^{\frac{1}{2}}}\\
&  =\int_{0}^{1}\left(  \int_{0}^{t}\left\|  \left|  f\right|  ^{\ast}%
(\cdot)\chi_{\lbrack0,s))}(\cdot)\right\|  _{\bar{X}}\frac{ds}{s\left(
\ln\frac{1}{s}\right)  ^{1-(k-1)/2}}\right)  \frac{dt}{t\left(  \ln\frac{1}%
{t}\right)  ^{\frac{1}{2}}}\\
&  +\int_{0}^{1}\left(  \int_{t}^{1}\left\|  \left|  f\right|  ^{\ast}%
(\cdot)\chi_{\lbrack0,t))}(\cdot)\right\|  _{\bar{X}}\frac{ds}{s\left(
\ln\frac{1}{s}\right)  ^{1-(k-1)/2}}\right)  \frac{dt}{t\left(  \ln\frac{1}%
{t}\right)  ^{\frac{1}{2}}}\\
&  =A+B.
\end{align*}
By Fubini's Theorem we have%
\begin{align*}
A  &  =\int_{0}^{1}\left(  \int_{0}^{t}\left\|  \left|  f\right|  ^{\ast
}(\cdot)\chi_{\lbrack0,s)}(\cdot)\right\|  _{\bar{X}}\frac{ds}{s\left(
\ln\frac{1}{s}\right)  ^{1-(k-1)/2}}\right)  \frac{dt}{t\left(  \ln\frac{1}%
{t}\right)  ^{\frac{1}{2}}}\\
&  =\int_{0}^{1}\left\|  \left|  f\right|  ^{\ast}(\cdot)\chi_{\lbrack
0,s)}(\cdot)\right\|  _{\bar{X}}\frac{1}{s\left(  \ln\frac{1}{s}\right)
^{1-(k-1)/2}}\left(  \int_{s}^{1}\frac{dt}{t\left(  \ln\frac{1}{t}\right)
^{\frac{1}{2}}}\right)  ds\\
&  =2\int_{0}^{1}\left\|  \left|  f\right|  ^{\ast}(\cdot)\chi_{\lbrack
0,s)}(\cdot)\right\|  _{\bar{X}}\frac{ds}{s\left(  \ln\frac{1}{s}\right)
^{1-k/2}}\\
&  =2\left\|  f\right\|  _{\bar{X}_{k,\log}}.
\end{align*}
We also have,
\begin{align*}
B  &  =\int_{0}^{1}\left(  \int_{t}^{1}\left\|  \left|  f\right|  ^{\ast
}(\cdot)\chi_{\lbrack0,t))}(\cdot)\right\|  _{\bar{X}}\frac{ds}{s\left(
\ln\frac{1}{s}\right)  ^{1-(k-1)/2}}\right)  \frac{dt}{t\left(  \ln\frac{1}%
{t}\right)  ^{\frac{1}{2}}}\\
&  =\int_{0}^{1}\left\|  \left|  f\right|  ^{\ast}(\cdot)\chi_{\lbrack
0,t))}(\cdot)\right\|  _{\bar{X}}\left(  \int_{t}^{1}\frac{ds}{s\left(
\ln\frac{1}{s}\right)  ^{1-(k-1)/2}}\right)  \frac{dt}{t\left(  \ln\frac{1}%
{t}\right)  ^{\frac{1}{2}}}\\
&  =\int_{0}^{1}\left\|  \left|  f\right|  ^{\ast}(\cdot)\chi_{\lbrack
0,s)}(\cdot)\right\|  _{\bar{X}}\frac{2}{k-1}\left(  \ln\frac{1}{t}\right)
^{(k-1)/2}\frac{dt}{t\left(  \ln\frac{1}{t}\right)  ^{\frac{1}{2}}}\text{
(since }k\geq2)\\
&  =\frac{2}{k-1}\int_{0}^{1}\left\|  \left|  f\right|  ^{\ast}(\cdot
)\chi_{\lbrack0,s)}(\cdot)\right\|  _{\bar{X}}\frac{dt}{t\left(  \ln\frac
{1}{t}\right)  ^{1-k/2}}\\
&  =\frac{2}{k-1}\left\|  f\right\|  _{\bar{X}_{k,\log}}.
\end{align*}
Now, $A+B$ gives (\ref{concluira}) concluding the proof of the theorem.
\end{proof}

In particular we have

\begin{example}
(cf. \cite{FKS} for the case $k=1$) Let $\bar{X}=$ $L^{p},$ then
$\underline{\alpha}_{L^{p}}=1/p.$ Let $M$ the smallest natural number such
that
\[
\frac{1}{p}>\frac{1}{M}.
\]
Let $n\geq M,$ and let $\Omega\subset\mathbb{R}^{n}$ be a bounded domain
normalized so that $\left|  \Omega\right|  =1.$ Then, for all $f\in
W_{0}^{k,p}(\Omega),$ we have
\[
\int_{0}^{1}\left(  \int_{0}^{s}\left(  \left|  f\right|  ^{\ast}(s)\right)
^{p}ds\right)  ^{1/p}\frac{dt}{t\left(  \ln\frac{1}{t}\right)  ^{1-k/2}}\leq
c(M,k,L^{p})\left\|  \left|  D^{k}f\right|  \right\|  _{L^{p}}.
\]
\end{example}

\subsection{The unit ball on $\mathbb{R}^{n}$}

Let $\left(  B^{n},\left|  \cdot\right|  ,\mu\right)  $ be the open unit ball
on $\mathbb{R}^{n}$ endowed with Euclidean metric $\left|  \cdot\right|  $ and
with the normalized Lebesgue measure $\mu=\frac{dx}{\gamma_{n}},$ where
$\gamma_{n}=\frac{\pi^{n/2}}{\Gamma(1+n/2)}$ is the measure of $B^{n}$. We
consider the Sobolev space $W^{1,1}:=W^{1,1}(B^{n})$ of functions $f$ $\in
L^{1}(B^{n})$ that are weakly differentiable on $B^{n}$ and $\left|  \nabla
f\right|  \in L^{1}(B^{n})$. Given $\bar{X}$ a r.i. space on $\left(
0,1\right)  $, the Sobolev space $W^{1,\bar{X}}:=W^{1,\bar{X}}(B^{n})$ is
defined by%
\[
W^{1,\bar{X}}=\left\{  f\in W^{1,1}:\left\|  f\right\|  _{W^{1,\bar{X}}%
}:=\left\|  \left|  f\right|  _{\mu}^{\ast}\right\|  _{\bar{X}}+\left\|
\left|  \nabla f\right|  _{\mu}^{\ast}\right\|  _{\bar{X}}<\infty\right\}  .
\]

Let $I_{B^{n}}$ be the isoperimetric profile of $\left(  B^{n},\left|
\cdot\right|  ,\mu\right)  .$ It is known that (cf. \cite[Lemma 1 pag
163]{Maz})
\begin{equation}
I_{B^{n}}(t)\geq\frac{\gamma_{n-1}}{\gamma_{n}}2^{1-1/n}\min(t,1-t)^{1-1/n}%
=J_{B^{n}}(t),\text{ \ \ }0<t<1. \label{isobola}%
\end{equation}
In fact, the constant that appears on the left hand side of (\ref{isobola}) is
best possible. Moreover, we recall that for $f\in W^{1,1}$ the inequalities
(\ref{aatot0}) and (\ref{aatot}) hold (cf. \cite{mamiast}).

Let%
\[
G(t)=\frac{1}{t\left(  \ln\frac{1}{t}\right)  ^{\frac{1}{2}}}.
\]
Then,%
\begin{align*}
\int_{0}^{1/2}\frac{t}{J_{B^{n}}(t)}G(t)dt  &  =\frac{\gamma_{n}}{\gamma
_{n-1}}\left(  \frac{1}{2}\right)  ^{1-1/n}\int_{0}^{1/2}t^{1/n}\frac
{dt}{t\left(  \ln\frac{1}{t}\right)  ^{\frac{1}{2}}}\\
&  \leq\frac{\gamma_{n}}{\gamma_{n-1}}\left(  \frac{1}{2}\right)  ^{1-1/n}%
\int_{0}^{1}t^{1/n}\frac{dt}{t\left(  \ln\frac{1}{t}\right)  ^{\frac{1}{2}}}\\
&  =\frac{\gamma_{n}}{\gamma_{n-1}}\left(  \frac{1}{2}\right)  ^{1-1/n}%
n^{\frac{1}{2}}\int_{0}^{\infty}z^{-\frac{1}{2}}e^{-z/n}dz\text{
\ \ }(e^{-z/n}=t^{1/n})\\
&  =\frac{\gamma_{n}}{\gamma_{n-1}}\left(  \frac{1}{2}\right)  ^{1-1/n}%
\sqrt{\pi}n^{\frac{1}{2}}.
\end{align*}
The associated isoperimetric operator is given by%
\[
Q_{J_{B^{n}}}f(t):=\frac{J_{B^{n}}(t)}{t}\int_{t}^{1/2}f(z)\frac{dz}{J_{B^{n}%
}(z)}=t^{-1/n}\int_{t}^{1/2}z^{1/n}f(z)\frac{dz}{z}=Q_{1/n}f(t).
\]
By the general theory (cf. (\ref{alcance}) and (\ref{norma}) in\ Section
\ref{back}), $Q_{1/n}$ is bounded on $\bar{X}$ if and only if $\underline
{\alpha}_{X}>1/n.$ Moreover,%
\[
\left\|  Q_{1/n}\right\|  \leq\int_{1}^{\infty}h_{\bar{X}}(\frac{1}%
{s})s^{\frac{1}{n}-1}ds.
\]
The previous discussion, combined with Theorems \ref{main} and \ref{inclusion}%
, gives the following

\begin{theorem}
Let $\bar{X}$ be a r.i. space on $(0,1).$ Then,

\begin{enumerate}
\item  If $\underline{\alpha}_{X}>0,$ then\footnote{note that $u_{\mu}^{\ast
}(s)=u^{\ast}(\gamma_{n}s)$.}, for all $f\in W^{1,\bar{X}}\ $%
\[
\int_{0}^{1/2}\left\|  \left(  f^{\ast}(\gamma_{n}s)-u^{\ast}(\gamma
_{n}t)\right)  \chi_{\lbrack0,t)}(s)\right\|  _{\bar{X}}\frac{dt}{t\left(
\ln\frac{1}{t}\right)  ^{\frac{1}{2}}}\leq\sqrt{\pi}n^{\frac{1}{2}}%
\frac{\gamma_{n}}{\gamma_{n-1}}\left(  \frac{1}{2}\right)  ^{^{1-1/n}}\left\|
Q\right\|  _{\bar{X}\rightarrow\bar{X}}\left\|  \left|  \nabla f\right|
^{\ast}(\gamma_{n}s)\right\|  _{\bar{X}}.
\]

\item  If $\overline{\alpha}_{X}<1,$ then, for all $f\in W^{1,\bar{X}}$%
\begin{align*}
&  \int_{0}^{1/2}\left\|  \left(  \frac{1}{s}\int_{0}^{s}f^{\ast}(\gamma
_{n}z)dz-f^{\ast}(\gamma_{n}s)\right)  \chi_{\lbrack0,t)}(s)\right\|
_{\bar{X}}\frac{dt}{t\left(  \ln\frac{1}{t}\right)  ^{\frac{1}{2}}}\\
&  \leq\sqrt{\pi}n^{\frac{1}{2}}\frac{\gamma_{n}}{\gamma_{n-1}}\left(
\frac{1}{2}\right)  ^{^{1-1/n}}\left\|  P\right\|  _{\bar{X}\rightarrow\bar
{X}}\left\|  \left|  \nabla f\right|  ^{\ast}(\gamma_{n}s)\right\|  _{\bar{X}%
}.
\end{align*}
\end{enumerate}

\begin{itemize}
\item [3.]Suppose that $\underline{\alpha}_{X}>0,$ and let $M$ be\ the
smallest natural number such that
\[
\underline{\alpha}_{\bar{X}}>1/M,
\]
and furthermore suppose that $n\geq M.$ Then, for all $f\in W^{1,\bar{X}},$ we
have%
\begin{align*}
&  \int_{0}^{1/2}\left\|  (f-med(f))^{\ast}(\gamma_{n}s)\chi_{\lbrack
0,t)}(s)\right\|  _{\bar{X}}\frac{dt}{t\left(  \ln\frac{1}{t}\right)
^{\frac{1}{2}}}\\
&  \leq\left(  \frac{\sqrt{\pi}n^{\frac{1}{2}}\gamma_{n}}{\gamma_{n-1}}\left(
\frac{1}{2}\right)  ^{^{1-1/n}}\int_{1}^{\infty}h_{X}(\frac{1}{s})s^{\frac
{1}{M}}\frac{ds}{s}\right)  \left\|  \left|  \nabla f\right|  ^{\ast}%
(\gamma_{n}s)\right\|  _{\bar{X}}.
\end{align*}
In particular, since
\[
\lim_{n\rightarrow\infty}\sqrt{\pi}n^{\frac{1}{2}}\frac{\gamma_{n}}%
{\gamma_{n-1}}\left(  \frac{1}{2}\right)  ^{^{1-1/n}}=\lim_{n\rightarrow
\infty}\pi\frac{n^{\frac{1}{2}}\Gamma(1+\frac{n-1}{2})}{\Gamma(1+\frac{n}{2}%
)}\left(  \frac{1}{2}\right)  ^{^{1-1/n}}=\frac{\sqrt{2}}{2}\pi,
\]
there exists a constant $c$ independent of $n,$ such that for all $f\in
W^{1,\bar{X}}$
\[
\int_{0}^{1/2}\left\|  (f-med(f))^{\ast}(\gamma_{n}s)\chi_{\lbrack
0,t)}(s)\right\|  _{\bar{X}}\frac{dt}{t\left(  \ln\frac{1}{t}\right)
^{\frac{1}{2}}}\leq c\left\|  \left|  \nabla f\right|  ^{\ast}(\gamma
_{n}s)\right\|  _{\bar{X}}.
\]
\end{itemize}
\end{theorem}

As a consequence we obtain the following (cf. \cite{BM})

\begin{corollary}
$\left(  B^{n},\left|  \cdot\right|  ,\mu\right)  $ is of Gaussian
isoperimetric type near zero.
\end{corollary}

\begin{proof}
By Corollary \ref{ma1} we have%
\[
t\left(  \ln\frac{1}{t}\right)  ^{1/2}-2t\left(  \ln2\right)  ^{1/2}=\int
_{t}^{1/2}\frac{1}{s\left(  \ln\frac{1}{s}\right)  ^{\frac{1}{2}}}ds\leq
CI_{B^{n}}(t),\text{ \ \ }0<t<1/2.
\]
Therefore, for $t\in(0,1/4)$%
\begin{align*}
t\left(  \ln\frac{1}{t}\right)  ^{1/2}  & \leq CI_{B^{n}}(t)+2\left(
\ln2\right)  ^{1/2}\\
& \leq CI_{B^{n}}(t)+\frac{1}{2}\left(  \ln\frac{1}{t}\right)  ^{1/2}%
\end{align*}
and the desired result follows.
\end{proof}

\subsection{The $n-$sphere}

Let $n\in\mathbb{N}$, $n\geq2,$ and let $\mathbb{S}^{n}$ be the unit sphere.
Consider the metric space ($\mathbb{S}^{n},d,\frac{dx_{n}}{\omega_{n}}),$
where $d$ is the geodesic distance, $dx_{n}$ is the Lebesgue measure on
$\mathbb{R}^{n}$ and $\omega_{n}=2\pi^{\frac{n+1}{2}}/\Gamma(\frac{n+1}{2}).$

Then, we have that for $f\in Lip(\mathbb{S}^{n})$ (cf. \cite[Proposition
1.5]{bobk}),%
\[
\left(  \int_{\mathbb{S}^{n}}\left|  f(x)-\int_{\mathbb{S}^{n}}fdx_{n}\right|
^{\frac{n}{n-1}}\frac{dx_{n}}{\omega_{n}}\right)  ^{\frac{n-1}{n}}\leq
\frac{\omega_{n}}{2\omega_{n-1}}\int_{\mathbb{S}^{n}}\left|  \nabla
f(x)\right|  \frac{dx_{n}}{\omega_{n}}.
\]
It follows that (cf. \cite{Maz})%
\[
I_{(\mathbb{S}^{n},d,\frac{dx_{n}}{\omega_{n}})}(t)\geq\frac{2\omega_{n-1}%
}{\omega_{n}}\min(t,1-t)^{1-1/n}=J_{\mathbb{S}^{n}}(t),\ \ 0<t<1.
\]

Consider the function%
\[
G(t)=\frac{1}{t\left(  \ln\frac{1}{t}\right)  ^{\frac{1}{2}}}%
\]
then%
\begin{align*}
\int_{0}^{1/2}\frac{t}{J_{\mathbb{S}^{n}}(t)}G(t)dt  &  =\frac{\omega_{n}%
}{2\omega_{n-1}}\int_{0}^{1/2}t^{1/n}\frac{dt}{t\left(  \ln\frac{1}{t}\right)
^{\frac{1}{2}}}\\
&  \leq\frac{\omega_{n}}{2\omega_{n-1}}\int_{0}^{1}t^{1/n}\frac{dt}{t\left(
\ln\frac{1}{t}\right)  ^{\frac{1}{2}}}\\
&  =\frac{\omega_{n}}{\omega_{n-1}}\sqrt{\pi}n^{\frac{1}{2}}.
\end{align*}
The isoperimetric operator in this case is given by%
\[
Q_{J_{\mathbb{S}^{n}}}f(t):=\frac{J_{\mathbb{S}^{n}}(t)}{t}\int_{t}%
^{1/2}f(z)\frac{dz}{J_{\mathbb{S}^{n}}(z)}=t^{-1/n}\int_{t}^{1/2}%
z^{1/n}f(z)\frac{dz}{z}=Q_{1/n}f(t).
\]
Therefore, $Q_{J_{\mathbb{S}^{n}}}$ is bounded on $\bar{X}$ if and only if
$\underline{\alpha}_{\bar{X}}>1/n.$ Moreover,%
\[
\left\|  Q_{1/n}\right\|  _{\bar{X}\rightarrow\bar{X}}\leq\int_{1}^{\infty
}h_{\bar{X}}(\frac{1}{s})s^{\frac{1}{n}-1}ds.
\]

Therefore, Theorems \ref{main} and \ref{inclusion} yield

\begin{theorem}
\label{esfera}Let $\bar{X}$ be a r.i. space on $(0,1)$.

\begin{enumerate}
\item  If $\underline{\alpha}_{\bar{X}}>0,$ then\footnote{Note that
$u_{\frac{dx_{n}}{\omega_{n}}}^{\ast}(s)=u_{dx_{n}}^{\ast}(\omega_{n}s)$.},
for all $f\in Lip(\mathbb{S}^{n})$%
\[
\int_{0}^{1/2}\left\|  \left(  f^{\ast}(\omega_{n}s)-f^{\ast}(\omega
_{n}t)\right)  \chi_{\lbrack0,t)}(s)\right\|  _{\bar{X}}\frac{dt}{t\left(
\ln\frac{1}{t}\right)  ^{\frac{1}{2}}}\leq\sqrt{\pi}n^{\frac{1}{2}}%
\frac{\omega_{n-1}}{\omega_{n}}\left\|  Q\right\|  _{\bar{X}\rightarrow\bar
{X}}\left\|  \left|  \nabla f\right|  ^{\ast}(\omega_{n}s)\right\|  _{\bar{X}%
}.
\]

\item  If $\overline{\alpha}_{\bar{X}}<1,$ then for all $f\in Lip(\mathbb{S}%
^{n})\ $%
\begin{align*}
&  \int_{0}^{1/2}\left\|  \left(  \frac{1}{s}\int_{0}^{s}f^{\ast}(\omega
_{n}z)dz-f^{\ast}(\omega_{n}t)\right)  \chi_{\lbrack0,t)}(s)\right\|
_{\bar{X}}\frac{dt}{t\left(  \ln\frac{\omega_{n}}{t}\right)  ^{\frac{1}{2}}}\\
&  \leq\sqrt{\pi}n^{\frac{1}{2}}\frac{\omega_{n-1}}{\omega_{n}}\left\|
P\right\|  _{\bar{X}\rightarrow\bar{X}}\left\|  \left|  \nabla f\right|
^{\ast}(\omega_{n}s)\right\|  _{\bar{X}}.
\end{align*}

\item  If $\underline{\alpha}_{\bar{X}}>0,$ let $M$ be\ the smallest natural
number such that
\[
\underline{\alpha}_{X}>1/M.
\]
Suppose that $n\geq M,$ then for all $f\in Lip(\mathbb{S}^{n}),$ we have
\begin{align*}
&  \int_{0}^{1/2}\left\|  (f-med(f))^{\ast}(\omega_{n}s)\chi_{\lbrack
0,t)}(s)\right\|  _{\bar{X}}\frac{dt}{t\left(  \ln\frac{1}{t}\right)
^{\frac{1}{2}}}\\
&  \leq\left(  \sqrt{\pi}n^{\frac{1}{2}}\frac{\omega_{n-1}}{\omega_{n}}%
\int_{1}^{\infty}h_{X}(\frac{1}{s})s^{\frac{1}{M}}\frac{ds}{s}\right)
\left\|  \left|  \nabla f\right|  ^{\ast}(\omega_{n}s)\right\|  _{\bar{X}}.
\end{align*}
In particular, since
\[
\lim_{n\rightarrow\infty}\frac{\omega_{n}}{\omega_{n-1}}\sqrt{\pi}n^{\frac
{1}{2}}=\lim_{n\rightarrow\infty}\pi\frac{n^{\frac{1}{2}}\Gamma(\frac{n}{2}%
)}{\Gamma(\frac{n+1}{2})}=\sqrt{2}\pi,
\]
there exists a constant $c$ independent of $n,$ such that for all $f\in
Lip_{X}(\mathbb{S}^{n}),$%
\[
\int_{0}^{1/2}\left\|  (f-med(f))^{\ast}(\omega_{n}s)\chi_{\lbrack
0,t)}(s)\right\|  _{\bar{X}}\frac{dt}{t\left(  \ln\frac{1}{t}\right)
^{\frac{1}{2}}}\leq c\left\|  \left|  \nabla f\right|  ^{\ast}(\omega
_{n}s)\right\|  _{\bar{X}}.
\]
\end{enumerate}
\end{theorem}

\subsection{Riemannian manifolds with positive curvature}

Let $V^{n}$ be a compact Riemannian manifold (without boundary). Let
$R(V^{n})$ denote the infimum over all the unit tangent vectors of $V^{n}$ of
the Ricci tensor, and let $I_{V^{n}}$ be the isoperimetric profile of the
manifold (with respect to the normalized Riemannian measure $d\sigma_{n}$). If
$R(V^{n})\geq(n-1)k>0$, then (cf. \cite{bobk})
\[
\sqrt{\frac{2k}{\pi}}\frac{\Gamma(\frac{n+1}{2})}{\Gamma(\frac{n}{2})}\left(
\int_{V^{n}}\left|  f(x)-\int_{V^{n}}fd\sigma_{n}\right|  ^{\frac{n}{n-1}%
}d\sigma_{n}\right)  ^{\frac{n-1}{n}}\leq\int_{V^{n}}\left|  \nabla
f(x)\right|  d\sigma_{n}\ \ (f\in Lip(V^{n})).
\]
As a consequence (cf. \cite{Maz}) the following isoperimetric inequality
holds
\[
I_{V^{n}}(t)\geq\sqrt{\frac{2k}{\pi}}\frac{\Gamma(\frac{n+1}{2})}{\Gamma
(\frac{n}{2})}\min(t,1-t)^{1-1/n}=J_{V^{n}}(t),\ \ 0<t<1.
\]

Theorems \ref{main} and \ref{inclusion} give:

\begin{theorem}
Let $\bar{X}$ be a r.i. space.

\begin{enumerate}
\item  If $\underline{\alpha}_{\bar{X}}>0,$ then for all $f\in Lip(V^{n}),$%
\[
\int_{0}^{\frac{1}{2}}\left\|  \left(  f_{\sigma_{n}}^{\ast}(s)-f_{\sigma_{n}%
}^{\ast}(t)\right)  \chi_{\lbrack0,t)}(s)\right\|  _{\bar{X}}\frac
{dt}{t\left(  \ln\frac{1}{t}\right)  ^{\frac{1}{2}}}\leq\frac{\pi}{\sqrt{2k}%
}\frac{n^{\frac{1}{2}}\Gamma(\frac{n}{2})}{\Gamma(\frac{n+1}{2})}\left\|
Q\right\|  _{\bar{X}\rightarrow\bar{X}}\left\|  \left|  \nabla f\right|
_{\sigma_{n}}^{\ast}\right\|  _{\bar{X}}.
\]

\item  If $\overline{\alpha}_{\bar{X}}<1,$ then for all $f\in Lip_{\bar{X}%
}(V^{n}),\ $%
\[
\int_{0}^{\frac{1}{2}}\left\|  \left(  f_{\sigma_{n}}^{\ast}(s)-f_{\sigma_{n}%
}^{\ast}(t)\right)  \chi_{\lbrack0,t)}(s)\right\|  _{\bar{X}}\frac
{dt}{t\left(  \ln\frac{1}{t}\right)  ^{\frac{1}{2}}}\leq\frac{\pi}{\sqrt{2k}%
}\frac{n^{\frac{1}{2}}\Gamma(\frac{n}{2})}{\Gamma(\frac{n+1}{2})}\left\|
P\right\|  _{\bar{X}\rightarrow\bar{X}}\left\|  \left|  \nabla f\right|
_{\sigma_{n}}^{\ast}\right\|  _{\bar{X}}.
\]

\item  If $\underline{\alpha}_{\bar{X}}>0,$ let $M$ be\ the smallest natural
number such that
\[
\underline{\alpha}_{X}>1/M.
\]
Suppose that $n\geq M.$ Then for all $f\in Lip(V^{n}),$ we have
\[
\int_{0}^{\frac{1}{2}}\left\|  (f-med(f))_{\sigma_{n}}^{\ast}(s)\chi
_{\lbrack0,t)}(s)\right\|  _{\bar{X}}\frac{dt}{t\left(  \ln\frac{1}{t}\right)
^{\frac{1}{2}}}\leq\left(  \frac{\pi}{\sqrt{2k}}\frac{n^{\frac{1}{2}}%
\Gamma(\frac{n}{2})}{\Gamma(\frac{n+1}{2})}\int_{1}^{\infty}h_{\bar{X}}%
(\frac{1}{s})s^{\frac{1}{M}}\frac{ds}{s}\right)  \left\|  \left|  \nabla
f\right|  _{\sigma_{n}}^{\ast}\right\|  _{\bar{X}}.
\]
In particular, since
\[
\lim_{n\rightarrow\infty}\frac{\pi}{\sqrt{2k}}\frac{n^{\frac{1}{2}}%
\Gamma(\frac{n}{2})}{\Gamma(\frac{n+1}{2})}=\lim_{n\rightarrow\infty}\pi
\frac{n^{\frac{1}{2}}\Gamma(\frac{n}{2})}{\Gamma(\frac{n+1}{2})}=\frac{\pi
}{\sqrt{k}},
\]
there exists a constant $c$ independent of $n,$ such that for all $f\in
Lip(V^{n})$%
\[
\int_{0}^{\frac{1}{2}}\left\|  (f-med(f))_{\sigma_{n}}^{\ast}(s)\chi
_{\lbrack0,t)}(s)\right\|  _{\bar{X}}\frac{dt}{t\left(  \ln\frac{1}{t}\right)
^{\frac{1}{2}}}\leq c\left\|  \left|  \nabla f\right|  _{\sigma_{n}}^{\ast
}\right\|  _{\bar{X}}.
\]
\end{enumerate}
\end{theorem}


\begin{thebibliography}{99}
\bibitem{alv}A. Alvino, Sulla diseguaglianza di Sobolev in spazi di Lorentz,
Boll. Un. Mat. Ital. \textbf{5}, 14-A (1977), 148-156.

\bibitem{astly}S. V. Astashkin and K. V. Lykov, \textsl{Strong extrapolation
spaces and interpolation}, Siberian Math. J. \textbf{50} (2009), 199-213.

\bibitem{bart}F. Barthe,\textsl{ Log-concave and spherical models in
isoperimetry}, Geom. Funct. Anal. \textbf{12} (2002), 32-55.

\bibitem{bl}J. Bergh and J. L\"{o}fstr\"{o}m, \textsl{Interpolation spaces. An
introduction}, Springer\-Verlag, Berlin-Heidelberg-New York, 1976{\footnotesize .}

\bibitem{BS}C. Bennett and R. Sharpley, \textsl{Interpolation of Operators},
Academic Press, Boston\textbf{, }1988.

\bibitem{bobk}S. G. Bobkov and C. Houdr\'{e}, \textsl{Some connections between
isoperimetric and Sobolev type inequalities}, Mem. Amer. Math. Soc.
\textbf{129} (1997), no 616.

\bibitem{boyd}D. W. Boyd, \textsl{Indices of function spaces and their
relationship to interpolation}, Canad. J. Math. \textbf{21} (1969), 1245--1254.

\bibitem{BM}Y. D. Burago and V. G. Mazja. \textsl{Certain questions of
potential theory and function theory for regions with irregular boundaries}.
Zap. Naucn. Sem. Leningrad. Otdel. Mat. Inst. Steklov. (LOMI) 3 (1967) 152.
(English translation: Potential theory and function theory for irregular
regions. Seminars in Mathematics. V. A. Steklov Mathematical Institute,
Leningrad, Vol. 3, Consultants Bureau, New York, 1969)

\bibitem{CNV}D. Cordero-Erausquin, B. Nazaret and C. Villani, \textsl{A
mass-transportation approach to sharp Sobolev and Gagliardo-Nirenberg
inequalities}, Adv. Math. \textbf{182} (2004) 307-332.

\bibitem{fio}A. Fiorenza, \textsl{Duality and reflexivity in grand Lebesgue
spaces}, Collect. Math. \textbf{51}, (2000), 131-148.

\bibitem{fioka}A. Fiorenza and G. E. Karadzhov, \textsl{Grand and Small
Lebesgue Spaces and their analogs}, Z. Anal. Anwendungen, \textbf{23} (2004), 657-681.

\bibitem{FKS}A. Fiorenza, M. Krbec and H. J. Schmeisser, \textsl{An improvement
of dimension-free Sobolev imbeddings in r.i. spaces}, preprint.

\bibitem{gar}A. Garsia and E. Rodemich, \textsl{Monotonicity of certain
functionals under rearrangements}, Ann. Inst. Fourier (Grenoble) \textbf{24}
(1974), 67-116.

\bibitem{griebel}M. Griebel, \textsl{Sparse grids and related approximation
schemes for higher dimensional problems}, in L. Pardo, A. Pinkus, E. Suli, and
M. Todd, editors, Foundations of Computational Mathematics (FoCM05),
Santander, pp 106-161, Cambridge University Press, 2006.

\bibitem{iw}T. Iwaniec and C. Sbordone, \textsl{On the integrability of the
Jacobian under minimal hypotheses}, Arch. Rational Mech. Anal. \textbf{119}
(1992), 129-143.

\bibitem{kami}G. E. Karadzhov and M. Milman, \textsl{Extrapolation Theory: New
Results and Applications, }J. Approx. Th. \textbf{133} (2005), 38-99.

\bibitem{krb1}M. Krbec and H. J. Schmeisser, \textsl{On dimension-free Sobolev
imbeddings I}, J. Math. Anal. Appl. \textbf{387} (2012), 114-125.

\bibitem{krb2}M. Krbec and H. J. Schmeisser, \textsl{On dimension-free Sobolev
imbeddings II}, Rev. Mat. Complutense \textbf{25} (2012), 247-265.

\bibitem{mamiadv}J. Martin and M. Milman, \textsl{Pointwise symmetrization
inequalities for Sobolev functions and applications}, Adv. Math. \textbf{225}
(2010), 121-199.

\bibitem{mamiast}J. Martin and M. Milman, \textsl{Fractional Sobolev
Inequalities: Symmetrization, Isoperimetry and Interpolation}, Preprint.

\bibitem{mamiproc}J. Martin and M. Milman, \textsl{Sobolev inequalities,
rearrangements, isoperimetry and interpolation spaces}, Contemporary
Mathematics \textbf{545} (2011), pp 167-193.

\bibitem{mmp}J. Martin, M. Milman and E. Pustylnik, \textsl{Sobolev
Inequalities: Symmetrization and Self Improvement via truncation}, J.
Funct.Anal.\textbf{ 252 }(2007), 677-695.

\bibitem{Maz}V.G. Maz'ya, \textsl{Sobolev Spaces}, Springer-Verlag, New York, 1985.

\bibitem{emi}E. Milman,  \textsl{A converse to the Maz'ya inequality for
capacities under curvature lower bound}, in A. Laptev (ed), Around the
research of Vladimir Maz'ya I: Function Spaces, Springer, 2010, pp 321-348.

\bibitem{emiro}E. Milman and L. Rotem, \textsl{Complemented Brunn-Minkowski
Inequalities and Isoperimetry for Homogeneous and Non-Homogeneous Measures} ( arXiv:1308.5695)

\bibitem{mifosil}M. Milman, \textsl{The computation of the }$K-$%
\textsl{functional for couples of rearrangement invariant spaces}, Resultate
Math. \textbf{5} (1982), 174-176.

\bibitem{rako}J. M. Rakotoson and B. Simon \textsl{Relative Rearrangement on a
finite measure space spplication to the regularity of weighted monotone
rearrangement (Part 1)}, Rev. R. Acad. Cienc. Exact. Fis. Nat. \textbf{91}
(1997), 17-31.

\bibitem{Ros}A. Ros, \textsl{The isoperimetric problem}, In: Global Theory of
Minimal Surfaces. Clay Math. Proc., vol. 2, pp. 175-209, Am. Math. Soc.,
Providence, 2005.

\bibitem{ta}G. Talenti, \textsl{Best constant in Sobolev inequality}, Ann.
Mat. Pura Appl. \textbf{110} (1976), 353-372.

\bibitem{tri}H. Triebel, \textsl{Tractable embeddings of Besov spaces into
Zygmund spaces}, Function spaces IX, 361-377, Banach Center Publ. \textbf{92},
Polish Acad. Sci. Inst. Math., Warsaw, 2011.

\bibitem{tri1}H. Triebel, \textsl{Tractable embeddings}, preprint, University
of Jena, Nov. 2012.
\end{thebibliography}
\end{document}